\documentclass[a4paper,11pt,reqno]{amsart}

\allowdisplaybreaks
\numberwithin{equation}{section}
\usepackage{mathtools}
\usepackage[shortlabels]{enumitem}
\usepackage{hyperref}
\usepackage[scriptsize,normal]{subfigure}
\usepackage[text={40pc,52pc},centering]{geometry}
\usepackage{placeins}

\newtheorem{thm}{Theorem}[section]
\newtheorem{lemma}[thm]{Lemma}

\theoremstyle{remark}
\newtheorem{remark}[thm]{Remark}

\newcommand{\real}{\mathbb R}
\DeclareMathOperator{\Var}{Var}
\DeclareMathOperator{\indicator}{\mathbf{1}}
\DeclareMathOperator{\hyper}{\sideset{_2}{_1}{\mathop F}}

\begin{document}

\title[Drift parameter estimation in the double mixed fractional Brownian model]{Drift parameter estimation in the double mixed fractional Brownian model via solutions of Fredholm equations with~singular kernels}

 \author{Yuliya Mishura}
 \address[Yu.\ Mishura]{Department of Probability, Statistics and Actuarial Mathematics, Taras Shevchenko National University of Kyiv, Volodymyrska 64, 01601 Kyiv, Ukraine}
 \email{yuliyamishura@knu.ua}
 \thanks{The first author is thankful to the University of Padova.}
 
 \author{Kostiantyn Ralchenko}
 \address[K.\ Ralchenko]{Department of Probability, Statistics and Actuarial Mathematics, Taras Shevchenko National University of Kyiv, Volodymyrska 64, 01601 Kyiv, Ukraine}
  \email{kostiantynralchenko@knu.ua}
 \address[K.\ Ralchenko]{School of Technology and Innovations, University of Vaasa, P.O.\ Box 700, FIN-65101 Vaasa, Finland}
 \email{kostiantyn.ralchenko@uwasa.fi}
 \thanks{The second author is supported by the Research Council of Finland, decision number 367468.}
 
 \author{Mykyta Yakovliev}
 \address[M.\ Yakovliev]{Department of Mathematical Analysis, Taras Shevchenko National University of Kyiv, Volodymyrska 64, 01601 Kyiv, Ukraine}
 \email{mykyta.yakovliev@knu.ua}%

 \begin{abstract}
We consider drift parameter estimation in a model driven by the sum of two independent fractional Brownian motions with different Hurst indices.
Although the maximum likelihood estimator (MLE) for this model is known theoretically, its computation requires solving an operator equation involving fractional covariance operators.
We develop an effective numerical method for approximating the solution of this equation by reformulating it as a Fredholm integral equation of the second kind with a weakly singular kernel.
The resulting algorithm enables practical computation of the MLE.
Numerical experiments illustrate the performance of the method.
 \end{abstract}
 

\maketitle

\section{Introduction}
Models involving sums of fractional Brownian motions arise naturally when stochastic dynamics exhibit multiple time scales or heterogeneous memory structures.
In this paper, we consider the model
\begin{equation}\label{eq:themodel}
X_t = \theta t + B^{H_1}_t + B^{H_2}_t, \quad t\ge 0,
\end{equation}
where $B^{H_1}$ and $B^{H_2}$ are two independent fractional Brownian motions with Hurst parameters\linebreak
$H_1, H_2 \in (\frac{1}{2},1)$.
We study the estimation of the drift parameter $\theta\in\real$ based on continuous-time observations of the trajectory $\{X_t, t\in[0,T]\}$.

Mixed models such as \eqref{eq:themodel}, involving fractional Brownian motions with different Hurst indices, are of considerable interest in applications, particularly in economics and financial mathematics, where the component with the smaller Hurst index typically represents short-term fluctuations, while the component with the larger index captures long-term dynamics.

The particular case $H_1=1/2$ corresponds to the mixed fractional Brownian motion introduced by Cheridito~\cite{cheridito}. Its probabilistic properties were studied in~\cite{Zili2006}, and various financial applications were considered in
\cite{Filatova2008,Sun2013,XZZZ2012,YanLuXia2025,ZXH2009}.
The model \eqref{eq:themodel} with arbitrary Hurst indices was investigated in~\cite{El-Nouty,Miao2008}, while financial applications were discussed in~\cite{HeChen}.
More general superpositions of fractional Brownian motions were studied in~\cite{Lechiheb2025,thale2009}. Recent developments include mixtures of infinitely many fractional Brownian motions~\cite{AlmaniSottinen2023}, limit theorems leading to Gaussian processes with such mixtures~\cite{BenderButkoDOvidioPagnini2024}, and time-changed versions of \eqref{eq:themodel} via inverse stable subordinators~\cite{Mliki2023}.

Statistical inference for mixed fractional models has also attracted considerable attention.
Drift estimation and filtering for mixed fractional Brownian motion were studied in~\cite{CCK2012}, and later extended to the case of two independent fractional Brownian motions in~\cite{Mishura2016}. Recent results include asymptotic inference for mixed fractional Brownian motion under high-frequency observations~\cite{Cai2026}.

In the present paper we focus on maximum likelihood estimation of the drift parameter $\theta$ in model \eqref{eq:themodel}. We develop an approximation method for computing the maximum likelihood estimator (MLE) introduced in \cite{MRStwofbm}, which is based on a general approach to drift estimation for Gaussian processes proposed in \cite{MRS-AJS}. In \cite{MRStwofbm} the MLE was thoroughly studied from a theoretical perspective: it was shown to be well defined and to possess properties such as strong and $L_2$-consistency and finite-sample normality.

However, from a practical viewpoint this estimator is not directly computable. It involves the $L_2$-solution of an operator equation whose existence and uniqueness are known, but for which no effective numerical procedure has been developed. The aim of the present paper is to provide such a procedure.

Our main idea is to rewrite the corresponding operator equation as a Fredholm integral equation of the second kind with a weakly singular kernel, for which efficient numerical methods are available. In particular, we employ the numerical approach developed in \cite{MRZ2024}, where kernels with similar singularities were treated. This method is based on a modification of the product-integration method for weakly singular kernels (see, e.g., \cite{Prem}) and on the algorithm described in \cite{Makogin21}, which relies on the modification proposed by Neta \cite{Neta}.

Another approach to drift estimation in model \eqref{eq:themodel} was proposed in \cite{Mishura2016}. It is also based on maximum likelihood estimation, but first transforms the observed process by the integral operator
\[
\widetilde{X}_t = C(H_1)\int_0^t s^{1/2-H_1} (t-s)^{1/2-H_1} dX_s.
\]
This approach can also be reduced to solving a Fredholm integral equation of the second kind, and a corresponding numerical method was developed in \cite{MRZ2024}. However, the resulting estimator involves integration with respect to the transformed observations $\widetilde X$, which themselves require integrating a weakly singular kernel with respect to the observed process. When discretized, this nested structure may introduce additional discretization errors and increases the numerical complexity of the procedure. In practical computations it is often preferable to work directly with the available observations of the process rather than to transform the data through additional integral operators. In contrast, the estimator proposed in the present paper involves integration directly with respect to the observed process. Although the corresponding kernel is more complicated, it is a deterministic object whose structure can be analysed in advance and efficiently approximated numerically. Moreover, its singular behaviour is similar to that appearing in \cite{MRZ2024}, which allows the use of analogous numerical techniques.

The paper is organized as follows.
In Section~\ref{sec:prelim} we recall the form of the MLE from \cite{MRStwofbm} together with its main properties. The estimator involves the solution $h_T$ of an operator equation expressed in terms of the covariance operators of fractional Brownian motions.
In Section~\ref{sec:fredholm} this equation is reformulated as a Fredholm integral equation with an explicitly given kernel $K(u,s)$. In Section~\ref{sec:K-hypergeo} we represent this kernel in terms of hypergeometric functions and provide a detailed analysis of its singularities along the diagonal $u=s$ and near the boundary of $[0,T]^2$. This representation allows us to apply efficient numerical methods for solving the equation.
Numerical approximation and simulation results are presented in Section~\ref{sec:simul}, where we discuss implementation issues and illustrate the performance of the MLE by means of simulations.

\section{Preliminaries}
\label{sec:prelim}

According to~\cite[Sect.~6.4.5]{MRStwofbm}, the MLE of the drift parameter $\theta$ in model \eqref{eq:themodel}, based on continuous-time observations of $X$ on $[0,T]$, is given by
\begin{equation*}
\widehat\theta_T
=
\frac{\int_0^T h_T(s)\, dX_s}{\int_0^T h_T(s)\, ds},
\end{equation*}
where the function $h_T$ is defined as the solution of the integral equation
\begin{equation}\label{eq:h-equation0}
\left(\Gamma_{H_1} + \Gamma_{H_2}\right) h_T
=
\indicator_{[0,T]}.
\end{equation}
Here, the operator $\Gamma_H$ is defined by
\begin{equation}\label{eq:Gamma}
(\Gamma_H f)(t)
=
H(2H-1)
\int_0^T \frac{f(s)}{|t-s|^{2-2H}}\,ds,
\qquad t\in(0,T).
\end{equation}

It was shown in \cite{MRStwofbm} that if
\begin{equation}\label{eq:H-assump}
H_1 \in (\tfrac{1}{2}, \tfrac{3}{4}]
\quad\text{and}\quad
H_2 \in (H_1,1),
\end{equation}
then:
\begin{enumerate}[$(i)$]
    \item equation~\eqref{eq:h-equation0} admits a unique solution
    $h_T \in L_2(0,T)$;
    \item the estimator $\widehat\theta_T$ is unbiased and
    $L_2$-consistent, strongly consistent, and normal with variance
    \begin{equation}
\Var \widehat\theta_T = 
\frac{1}{\int_0^T h_T(t) \, dt},
\label{eq-prop-contt-var}
\end{equation}
    see \cite[Thm.~6.6]{MRStwofbm}.
\end{enumerate}

However, no explicit expression for the solution $h_T$ of
equation~\eqref{eq:h-equation0} is currently available.
The main objective of this paper is to develop an effective method for approximating the solution of this equation, thereby enabling the numerical computation of the estimator $\widehat\theta_T$.
The key idea is to reformulate equation~\eqref{eq:h-equation0} as a
Fredholm integral equation with a weakly singular kernel, for which
well-established approximation methods exist.

\section{Equation for \texorpdfstring{$h_T$}{h} as a Fredholm equation of the second kind}
\label{sec:fredholm}

Let us introduce the following notation:
\begin{align}
\psi(u,s)
&\coloneqq
\frac{\partial}{\partial u}\left(\int_0^{u} 
(u-t)^{\frac12-H_1} t^{\frac12-H_1}(s-t)^{2H_2-2}\,dt\right),
\;\; 0<u<s<T,
\label{eq:psi}
\\
\varphi(u,s) &\coloneqq
\frac{\partial}{\partial u}\biggl(
\int_s^u (u-t)^{\frac12-H_1}\, t^{\frac12-H_1}\, (t-s)^{2H_2-2}\,dt
\biggr), \;\; 0<s<u<T,
\label{eq:phi}
\\
\tau(u,s) &\coloneqq
\frac{\partial}{\partial u}\biggl(
\int_0^{s}
(u-t)^{\frac12-H_1} t^{\frac12-H_1}(s-t)^{2H_2-2}\,dt
\biggr), \quad 0<s<u<T,
\label{eq:tau}
\\
\rho(u,s) &\coloneqq \varphi(u,s)+\tau(u,s), 
\quad 0<s<u<T.
\label{eq:rho}
\end{align}
These functions may be represented in terms of hypergeometric and beta functions, see Section \ref{sec:K-hypergeo}.

\begin{thm}\label{thm:fredholm}
Assume that $H_1$ and $H_2$ satisfy \eqref{eq:H-assump}.
Then the function $h_T$ satisfies the following Fredholm integral equation of the second kind:
\begin{equation}\label{Fred_mle}
h_T(u) + c(H_1,H_2)\, \int_0^{T} h_T(s)\,K(u,s)\,ds = g_T(u),
\quad \text{for a.a. } u \in (0,T),
\end{equation}
with
\begin{equation}
c(H_1,H_2)
= \frac{\Gamma(2-2H_1)\cos\bigl(\pi(1-H_1)\bigr)H_2(2H_2-1)}
     {\pi\,H_1(2H_1-1)\,\Gamma^2(\frac32-H_1)},
\quad
g_T(u)
= \frac{u^{1/2-H_1}(T-u)^{1/2-H_1}}
{2H_1\,B\left(\tfrac{3}{2}-H_1,\,H_1+\tfrac{1}{2}\right)},
\label{eq:rhs}
\end{equation}
where $\Gamma$ and $B$ denote the gamma and beta functions respectively,
and
\[
K (u,s) = 
\begin{cases}
K_{-} (u,s), & 0 < s < u < T,\\
K_{+} (u,s), & 0 < u < s < T.
\end{cases}
\]
Here
\[
K_{-}(u,s) = u^{H_1-\frac12}(T-u)^{\frac12-H_1}\rho(u,s)
+(H_1-\tfrac12)\,u^{\frac12-H_1}
\int_u^T
\frac{u^{2H_1-1}\rho(u,s) - t^{2H_1-1}\rho(t,s)}
{(t-u)^{H_1+\frac12}}\,dt,
\]
and
\begin{align*}
K_{+}(u,s)
&= u^{H_1-\frac12}(s-u)^{\frac12-H_1}\psi(u,s)
+ (H_1-\tfrac12)\,u^{\frac12-H_1}
\int_u^s
\frac{
u^{2H_1-1}\,\psi(u,s) - t^{2H_1-1}\,\psi(t,s)
}
{(t-u)^{H_1+\frac12}}\,dt
\\
&\quad-(H_1-\tfrac12)u^{\frac12-H_1}
\int_{s}^T
(t-u)^{-\frac12-H_1}\,t^{2H_1-1}
\rho(t,s)\,dt.
\end{align*}
\end{thm}

\begin{proof}
It is well known that for any $H\in(\tfrac12,1)$, the operator $\Gamma_H$, defined by \eqref{eq:Gamma}, is a bounded and injective
linear operator from $L_2(0,T)$ to $L_2(0,T)$.
This follows from its representation in terms of Riemann--Liouville fractional integrals,
\[
\Gamma_H
=
H\Gamma(2H)\bigl(I_{0+}^{2H-1} + I_{T-}^{2H-1}\bigr),
\]
together with the boundedness and injectivity of the operators
$I_{0+}^{\alpha}$ and $I_{T-}^{\alpha}$ for $\alpha>0$; see Appendix~\ref{app:fraccalc}
and also~\cite[Sect.~6.6]{MRStwofbm}.

Moreover, the range of $\Gamma_H$ contains the indicator function $\indicator_{[0,T]}$.
More precisely,
\begin{equation}\label{eq:g}
\Gamma_H g_T = \indicator_{[0,T]},
\qquad
g_T(u)
=
\frac{u^{\frac12-H}(T-u)^{\frac12-H}}
{2H\,B\!\left(\tfrac{3}{2}-H,\,H+\tfrac{1}{2}\right)},
\end{equation}
see~\cite[Prop.~2.1]{NVV99}.

Consequently, equation~\eqref{eq:h-equation0} can be rewritten in the equivalent form
\begin{equation}\label{eq:h-equation}
\left(I + \Gamma_{H_1}^{-1}\Gamma_{H_2}\right) h_T
=
\Gamma_{H_1}^{-1}\indicator_{[0,T]}.
\end{equation}

Now we transform the equation \eqref{eq:h-equation} into the form \eqref{Fred_mle}.
The representation \eqref{eq:rhs} for the right-hand side
$g_T(u) = (\Gamma_{H_1}^{-1}\indicator_{[0,T]})(u)$
of this equation follows from \eqref{eq:g}.
So it remains to obtain the claimed representation for the term $(\Gamma_{H_1}^{-1}\Gamma_{H_2} h_T)(u)$. This will be done in several steps.

\emph{Step~1. Representation of $\Gamma_{H_1}^{-1}\Gamma_{H_2} h_T$ in terms of fractional derivatives.}
Recall that $h_T$ is an $L_2(0,T)$-solution of the equation~\eqref{eq:h-equation}.
Set
\[
y \coloneqq \Gamma_{H_1}^{-1}\Gamma_{H_2} h_T.
\]
By~\cite[Thm.~6.8]{MRStwofbm}, under assumptions~\eqref{eq:H-assump}, the operator
$\Gamma_{H_1}^{-1}\Gamma_{H_2}\colon L_2(0,T)\to L_2(0,T)$
is compact and defined on the entire space $L_2(0,T)$.
Consequently, $y\in L_2(0,T)\subset L_1(0,T)$.
Moreover, $\Gamma_{H_1} y=\Gamma_{H_2} h_T$, that is, $y$ solves the integral equation
\begin{equation}\label{eq:y-equation}
\int_0^T \frac{y(s)\,ds}{|u-s|^p}=f(u)
\quad \text{for almost all } u\in(0,T),
\end{equation}
with $p=2-2H_1$ and
\begin{equation}\label{eq:f-def}
f(u)
=
\frac{(\Gamma_{H_2}h_T)(u)}{H_1(2H_1-1)}
=
\frac{H_2(2H_2-1)}{H_1(2H_1-1)}
\int_0^T \frac{h_T(s)}{|u-s|^{2-2H_2}}\,ds.
\end{equation}
Note that $f\in L_2(0,T)$, since $\Gamma_H$ is a bounded operator from $L_2(0,T)$ to $L_2(0,T)$ and $h_T\in L_2(0,T)$.

According to~\cite[Theorem~6.7]{MRStwofbm}, any solution
$y\in L_1(0,T)$ of equation~\eqref{eq:y-equation} admits the representation
\begin{equation}\label{ch3:MRS:eq:solieq1}
y(u)
=
\frac{\Gamma(p)\cos\!\left(\frac{\pi p}{2}\right)}{\pi\,u^{(1-p)/2}}
\,
\mathcal{D}^{(1-p)/2}_{T-}
\left(
z^{1-p}
\mathcal{D}^{(1-p)/2}_{0+}
\left(\frac{f(v)}{v^{(1-p)/2}}\right)\Bigg|_{v=z}
\right)\Bigg|_{z=u}
\end{equation}
for almost all $u\in(0,T)$.
Here $\mathcal{D}^{\alpha}_{a+}$ and $\mathcal{D}^{\alpha}_{b-}$ denote the
Riemann--Liouville fractional derivatives; see Appendix~\ref{app:fraccalc}.

Substituting~\eqref{eq:f-def} into~\eqref{ch3:MRS:eq:solieq1} with $p=2-2H_1$, we obtain
\begin{equation}\label{eq:GinvG-in-H-const-out}
(\Gamma_{H_1}^{-1}\Gamma_{H_2}h_T)(u)
=
C(H_1,H_2)\,
u^{\frac12-H_1}
\mathcal{D}^{H_1-\frac12}_{T-}
\left(
z^{2H_1-1}
\mathcal{D}^{H_1-\frac12}_{0+}
\left(
v^{\frac12-H_1}
\int_0^T \frac{h_T(s)}{|v-s|^{\,2-2H_2}}\,ds
\right)\Bigg|_{v=z}
\right)\Bigg|_{z=u},
\end{equation}
for $u\in(0,T)$, where
\[
C(H_1,H_2)
\coloneqq
\frac{\Gamma(2-2H_1)\cos\!\bigl(\pi(1-H_1)\bigr)H_2(2H_2-1)}
{\pi\,H_1(2H_1-1)}.
\]

Representation~\eqref{eq:GinvG-in-H-const-out} enables us to rewrite the operator
$\Gamma_{H_1}^{-1}\Gamma_{H_2}$ as an integral operator with an explicit kernel.
To this end, define
\[
I_1(v) := \int_0^v \frac{h_T(s)}{(v-s)^{2-2H_2}}\,ds,
\qquad
I_2(v) := \int_v^T \frac{h_T(s)}{(s-v)^{2-2H_2}}\,ds,
\]
so that
\[
\int_0^T \frac{h_T(s)}{|v-s|^{\,2-2H_2}}\,ds = I_1(v)+I_2(v).
\]

Substituting this into~\eqref{eq:GinvG-in-H-const-out} and using linearity
of the fractional derivatives, we obtain
\[
(\Gamma_{H_1}^{-1}\Gamma_{H_2}h_T)(u)
= C(H_1,H_2)\,\bigl(A_1(u) + A_2(u)\bigr),
\]
where
\begin{align}
A_1(u) &:= u^{\frac12-H_1}\,
\mathcal{D}^{H_1-\frac12}_{T-}
\left(
z^{2H_1-1}\,
\mathcal{D}^{H_1-\frac12}_{0+}
\left(v^{\frac12-H_1}I_1(v)\right)\Big|_{v=z}
\right)\Big|_{z=u},\label{eq:A1-def}
\\
A_2(u) &:= u^{\frac12-H_1}
\mathcal{D}^{H_1-\frac12}_{T-}
\left(
z^{2H_1-1}\,
\mathcal{D}^{H_1-\frac12}_{0+}
\left(
v^{\frac12-H_1} I_2(v)
\right)\Big|_{v=z}
\right)\Big|_{z=u}.\label{eq:A2-def}
\end{align}

In the subsequent analysis we will treat the contributions $A_1(u)$ and $A_2(u)$
separately.

\emph{Step 2. Transformation of $A_1(u)$.}
By definition of the Riemann--Liouville fractional derivative,
we get
\[
\mathcal{D}^{H_1-\frac12}_{0+}
\Bigl(
v^{\frac12-H_1} I_1(v)
\Bigr)\Big|_{v=u}
=\frac{1}{\Gamma\bigl(\frac32-H_1\bigr)}\frac{d}{du}
\int_0^u (u-t)^{\frac12-H_1} t^{\frac12-H_1}
\biggl(\int_0^t \frac{h_T(s)}{(t-s)^{2-2H_2}}\,ds\biggr) dt.
\]

We now change the order of integration and differentiate w.r.t.\ $u$:
\begin{align*}
\mathcal{D}^{H_1-\frac12}_{0+}
\Bigl(
v^{\frac12-H_1} I_1(v)
\Bigr)\Big|_{v=u}
&=\frac{1}{\Gamma\bigl(\frac32-H_1\bigr)}\frac{d}{du}
\int_0^u h_T(s)
\biggl(
\int_s^u
(u-t)^{\frac12-H_1} t^{\frac12-H_1} (t-s)^{2H_2-2}
\,dt
\biggr) ds
\\
&=\frac{1}{\Gamma\bigl(\frac32-H_1\bigr)}
\int_0^u h_T(s)\,
\frac{\partial}{\partial u}
\biggl(
\int_s^u
(u-t)^{\frac12-H_1} t^{\frac12-H_1} (t-s)^{2H_2-2}
\,dt
\biggr) ds
\\
&= \frac{1}{\Gamma\!\bigl(\tfrac32-H_1\bigr)}
\int_0^u h_T(s)\,
\varphi(u,s) ds,
\end{align*}
where $\varphi(u,s)$ is defined by \eqref{eq:phi}.

Thus, the term $A_1(u)$ defined by \eqref{eq:A1-def}, takes the form
\[
A_1(u)
= \frac{1}{\Gamma(\frac32-H_1)}\,
u^{\frac12-H_1}\,
\mathcal{D}^{H_1 - \frac12}_{T-}
\Bigl(
z^{2H_1-1}\,
\int_0^z h_T(s)\,\varphi(z,s)\,ds
\Bigr)\Bigr|_{z=u}.
\]

Applying the definition of the Riemann--Liouville fractional derivative,
we obtain
\[
A_1(u) =
-\frac{u^{\frac12-H_1}}{\Gamma^2(\frac32-H_1)}\,
\frac{d}{du}
\left[\int_u^T
(t-u)^{\frac12-H_1}\,
t^{2H_1-1}
\left(
\int_0^t h_T(s)\,\varphi(t,s)\,ds
\right) dt\right].
\]
Changing the order of integration, we get
\begin{align}
A_1(u) &= -\frac{u^{\frac12-H_1}}{\Gamma^2(\frac32-H_1)}
\frac{d}{du}
\left[
\int_0^T h_T(s)
\biggl(
\int_{u\vee s}^T
(t-u)^{\frac12-H_1}\,
t^{2H_1-1}\,
\varphi(t,s)\,dt
\biggr) ds
\right]
\notag\\
&=
-\frac{u^{\frac12-H_1}}{\Gamma^2(\frac32-H_1)}
\frac{d}{du}
\Biggl[
\int_0^u h_T(s)
\biggl(
\int_{u}^T
(t-u)^{\frac12-H_1}\,
t^{2H_1-1}\,
\varphi(t,s)\,dt
\biggr) ds
\notag\\*
&\qquad\qquad\qquad\qquad\qquad+
\int_u^T h_T(s)
\biggl(
\int_{s}^T
(t-u)^{\frac12-H_1}\,
t^{2H_1-1}\,
\varphi(t,s)\,dt
\biggr) ds
\Biggr].
\label{eq:A1-split}
\end{align}

For convenience, denote
\begin{align*}
F_1(u,s)
&\coloneqq
\int_{u}^T
(t-u)^{\frac12-H_1}\,
t^{2H_1-1}\,
\varphi(t,s)\,dt,
\qquad 0\le s\le u,
\\
F_2(u,s)
&\coloneqq
\int_{s}^T
(t-u)^{\frac12-H_1}\,
t^{2H_1-1}\,
\varphi(t,s)\,dt,
\qquad u\le s\le T.
\end{align*}
We now differentiate the bracket in \eqref{eq:A1-split} with respect to $u$:
\begin{align*}
\frac{d}{du}\int_0^u h_T(s)\,F_1(u,s)\,ds
&= h_T(u)\,F_1(u,u)
+ \int_0^u h_T(s)\,\frac{\partial}{\partial u}F_1(u,s)\,ds,
\\
\frac{d}{du}\int_u^T h_T(s)\,F_2(u,s)\,ds
&= -\,h_T(u)\,F_2(u,u)
+
\int_u^T h_T(s)\,\frac{\partial}{\partial u}F_2(u,s)\,ds.
\end{align*}
Substituting this into \eqref{eq:A1-split}, we obtain
\begin{multline*}
A_1(u)
=
-\frac{1}{\Gamma^2\bigl(\tfrac32-H_1\bigr)}\,
u^{\frac12-H_1}\,
\Biggl[
h_T(u)\,F_1(u,u)
- h_T(u)\,F_2(u,u)
\\
+ \int_0^u h_T(s)\,\frac{\partial}{\partial u}F_1(u,s)\,ds
+ \int_u^T h_T(s)\,\frac{\partial}{\partial u}F_2(u,s)\,ds
\Biggr].
\end{multline*}

Since $F_1(u,u)$ and $F_2(u,u)$ coincide, we arrive at
\begin{align}
A_1(u)
&=
-\frac{u^{\frac12-H_1}}{\Gamma^2(\frac32-H_1)}
\Biggl[
\int_0^u h_T(s)\,\frac{\partial}{\partial u}F_1(u,s)\,ds
+
\int_u^T h_T(s)\,\frac{\partial}{\partial u}F_2(u,s)\,ds
\Biggr].
\label{eq:A1viaF1F2}
\end{align}

Now, let us compute $\frac{\partial}{\partial u}F_1(u,s)$ and $\frac{\partial}{\partial u}F_2(u,s)$.
The differentiation of $F_2(u,s)$ is straightforward:
\begin{equation}\label{eq:dF2}
\frac{\partial}{\partial u}F_2(u,s)
=
-(\tfrac12-H_1)
\int_s^T
(t-u)^{-\frac12-H_1}\,
t^{2H_1-1}\,\varphi(t,s)\,dt,
\qquad u\le s\le T,
\end{equation}

In order to find $\frac{\partial}{\partial u}F_1(u,s)$, 
we note that it can be represented as a fractional derivative (see \eqref{eq:derright} in the Appendix \ref{app:fraccalc}):
\[
\frac{\partial}{\partial u}F_1(u,s)
=
\frac{d}{du}
\int_u^T (t-u)^{\frac12-H_1} f_s(t)\,dt
=
-\Gamma(\tfrac32-H_1)\,
\left(\mathcal{D}^{H_1 - \frac12}_{T-} f_s\right)(u).
\]
where
\[
f_s(t) := t^{2H_1-1} \varphi(t,s).
\]

Now we apply the Weyl-type representation \eqref{eq:derright-weyl} of the right-sided fractional
derivative on $(0,T)$.
We get
\begin{equation}\label{eq:dF1}
\frac{\partial}{\partial u}F_1(u,s)
=
-\frac{u^{2H_1-1}\,\varphi(u,s)}{(T-u)^{H_1-\frac12}}
-(H_1-\tfrac12)
\int_u^T
\frac{u^{2H_1-1}\,\varphi(u,s) - t^{2H_1-1}\,\varphi(t,s)}
     {(t-u)^{H_1+\frac12}}\,dt.
\end{equation}

Substituting \eqref{eq:dF2} and \eqref{eq:dF1} into \eqref{eq:A1viaF1F2},
we obtain
\[
A_1(u)
=\frac{1}{\Gamma^2(\frac32-H_1)}\,
\int_0^T h_T(s) K_1(u,s) ds
\]
with
\[
K_1 (u,s) = 
\begin{cases}
K_{1,-} (u,s), & 0 < s < u,\\
K_{1,+} (u,s), & u < s < T,
\end{cases}
\]
where
\begin{align*}
K_{1,-} &=
u^{H_1-\frac12}(T-u)^{\frac12-H_1}\varphi(u,s)
+(H_1-\tfrac12)u^{\frac12-H_1}
\int_u^T
\frac{u^{2H_1-1}\,\varphi(u,s) - t^{2H_1-1}\,\varphi(t,s)}
     {(t-u)^{H_1+\frac12}}\,dt,
\\
K_{1,+} &= -(H_1 - \tfrac12)
u^{\frac12-H_1}\int_s^T
(t-u)^{-\frac12-H_1}\,
t^{2H_1-1}\,\varphi(t,s)\,dt.
\end{align*}

\emph{Step 3. Transformation of $A_2(u)$.}
We now perform analogous computations for the term $A_2(u)$ defined by \eqref{eq:A2-def}.

By the definition of the Riemann--Liouville left-sided fractional derivative,
we obtain
\[
\mathcal{D}^{H_1-\frac12}_{0+}
\Bigl(
u^{\frac12-H_1} I_2(u)
\Bigr)
=\frac{1}{\Gamma(\frac32-H_1)}\frac{d}{du}
\int_0^u (u-t)^{\frac12-H_1} t^{\frac12-H_1}
\biggl(\int_t^T \frac{h_T(s)}{(s-t)^{2-2H_2}}\,ds\biggr) dt.
\]

We now change the order of integration:
\begin{align}
\mathcal{D}^{H_1-\frac12}_{0+}
\Bigl(
u^{\frac12-H_1} I_2(u)
\Bigr)
&=
\frac{1}{\Gamma(\frac32-H_1)}\frac{d}{du}
\Biggl[
\int_0^u h_T(s)
\biggl(
\int_0^{s}
(u-t)^{\frac12-H_1} t^{\frac12-H_1}(s-t)^{2H_2-2}\,dt
\biggr) ds
\notag\\*
&\qquad\qquad\qquad\qquad
+
\int_u^T h_T(s)
\biggl(
\int_0^{u}
(u-t)^{\frac12-H_1} t^{\frac12-H_1}(s-t)^{2H_2-2}\,dt
\biggr) ds
\Biggr]
\notag\\
&=
\frac{1}{\Gamma(\frac32-H_1)}\frac{d}{du}
\Biggl[
\int_0^u h_T(s)\,G_1(u,s)\,ds
+
\int_u^T h_T(s)\,G_2(u,s)\,ds
\Biggr],
\label{eq:D0+A2-Fubini-split}
\end{align}
where
\begin{align}
G_1(u,s)
&:=
\int_0^{s}
(u-t)^{\frac12-H_1} t^{\frac12-H_1}(s-t)^{2H_2-2}\,dt,
\quad 0\le s\le u,
\notag\\
G_2(u,s)
&:=
\int_0^{u}
(u-t)^{\frac12-H_1} t^{\frac12-H_1}(s-t)^{2H_2-2}\,dt,
\quad u\le s\le T.
\label{eq:G2}
\end{align}

Let us differentiate the expression inside the derivative in \eqref{eq:D0+A2-Fubini-split}:
\begin{align*}
\frac{d}{du}\int_0^u h_T(s)\,G_1(u,s)\,ds
&=
h_T(u)\,G_1(u,u)
+
\int_0^u h_T(s)\,\frac{\partial}{\partial u}G_1(u,s)\,ds.
\\
\frac{d}{du}\int_u^T h_T(s)\,G_2(u,s)\,ds
&=
-\,h_T(u)\,G_2(u,u)
+
\int_u^T h_T(s)\,\frac{\partial}{\partial u}G_2(u,s)\,ds.
\end{align*}

Since
$G_1(u,u) = G_2(u,u)$
for all $u\in(0,T)$,
the boundary contribution vanishes.
Moreover, according to the notation \eqref{eq:psi} and \eqref{eq:tau},
\[
\frac{\partial}{\partial u}G_1(u,s) = \tau(u,s)
\quad\text{and}\quad
\frac{\partial}{\partial u}G_2(u,s) = \psi(u,s).
\]
Therefore,
\[
\mathcal{D}^{H_1-\frac12}_{0+}
\Bigl(
u^{\frac12-H_1} I_2(u)
\Bigr)
= \frac{1}{\Gamma(\frac32-H_1)}
\Biggl[
\int_0^u h_T(s)\,\tau(u,s)\,ds
+
\int_u^T h_T(s)\,\psi(u,s)\,ds
\Biggr].
\]

Substituting the expression for the inner derivative into the definition of $A_2(u)$, we obtain
\[
A_2(u)
=
\frac{u^{\frac12-H_1}}{\Gamma(\frac32-H_1)}
\mathcal{D}^{H_1-\frac12}_{T-}
\Bigl(
u^{2H_1-1}
\Bigl[
\int_0^u h_T(s) \,\tau(u,s)\, ds
+\int_u^T h_T(s)\,\psi(u,s)\,ds
\Bigr]
\Bigr).
\]

Applying the definition of the right-sided fractional derivative, we get
\begin{align*}
A_2(u) &=
-\frac{u^{\frac12-H_1}}{\Gamma^2(\frac32-H_1)}\,
\frac{d}{du}\Biggl(
\int_u^T
(t-u)^{\frac12-H_1}\,t^{2H_1-1}
\Biggl[
\int_0^t h_T(s)
\tau(t,s) ds
+
\int_t^T h_T(s)\,\psi(t,s)\,ds
\Biggr] dt\Biggr)
\\
&\eqqcolon-\frac{1}{\Gamma^2(\frac32-H_1)}\,
u^{\frac12-H_1}\,
\frac{d}{du}\Bigl(J_1(u) + J_2(u)\Bigr).
\end{align*}

Let us consider first the term $J_1$ and change the order of integration in it.
We obtain
\begin{align}
J_1(u) &= 
\int_0^u h_T(s)
\int_{u}^T
(t-u)^{\frac12-H_1}\,t^{2H_1-1}\,\tau(t,s)
\,dt\,ds
+\int_u^T h_T(s)
\int_{s}^T
(t-u)^{\frac12-H_1}\,t^{2H_1-1} \,\tau(t,s)
\,dt\,ds
\notag\\
&=
\int_0^u h_T(s) Q_1(u,s)\,ds
+\int_u^T h_T(s) Q_2(u,s)\,ds,
\label{eq:J1}
\end{align}
where
\[
Q_1(u,s) = \int_{u}^T
(t-u)^{\frac12-H_1}\,t^{2H_1-1}\,\tau(t,s)
\,dt,
\qquad
Q_2(u,s) = \int_{s}^T
(t-u)^{\frac12-H_1}\,t^{2H_1-1} \,\tau(t,s)\,dt.
\]

Differentiating both terms in the right-hand side of \eqref{eq:J1}, we get
\begin{align*}
\frac{d}{du} \int_0^u h_T(s) Q_1(u,s)\,ds
&=
h_T(u)Q_1(u,u)
+\int_0^u h_T(s)
\frac{\partial}{\partial u} Q_1(u,s)\,ds,
\\
\frac{d}{du} \int_u^T h_T(s) Q_2(u,s)\,ds
&=
-h_T(u) Q_2(u,u)
+\int_u^T h_T(s)
\frac{\partial}{\partial u}Q_2(u,s)\,ds.
\end{align*}
Note that $Q_1(u,u) = Q_2 (u,u)$. Hence,
\begin{align}
\frac{d}{du}J_1(u)
=\int_0^u h_T(s)
\frac{\partial}{\partial u} Q_1(u,s)\,ds
+\int_u^T h_T(s)
\frac{\partial}{\partial u}Q_2(u,s)\,ds.
\label{eq:dJ1+dJ2}
\end{align}

The term $Q_2(u,s)$ can be differentiated directly:
\begin{equation}\label{eq:dQ2}
\frac{\partial}{\partial u}Q_2(u,s) = (H_1-\tfrac12)\int_{s}^T
(t-u)^{-\frac12-H_1} t^{2H_1-1} \tau(t,s)\,dt.
\end{equation}

The derivative $\frac{\partial}{\partial u} Q_1(u,r)$ can be represented as a right-sided
Riemann--Liouville fractional derivative:
\[
\frac{\partial}{\partial u} Q_1(u,r)
=
-\,\Gamma\!\bigl(\tfrac32-H_1\bigr)\,
\mathcal{D}^{\,H_1-\frac12}_{T-}
\bigl( t^{2H_1-1}\tau(t,s))\big|_{t=u}.
\]
Applying the Weyl representation \eqref{eq:derright-weyl} of a fractional derivative, we obtain
\begin{equation}
\frac{\partial}{\partial u} Q_1(u,r)
= -\frac{u^{2H_1-1}\tau(u,s)}{(T-u)^{H_1-\frac12}}
-(H_1-\tfrac12)
\int_u^T
\frac{
u^{2H_1-1}\tau(u,s)
-
t^{2H_1-1}\tau(t,s)
}
{(t-u)^{H_1+\frac12}}\,dt.
\label{eq:dQ1}
\end{equation}

Inserting \eqref{eq:dQ1} and \eqref{eq:dQ2} into \eqref{eq:dJ1+dJ2} we arrive at:
\begin{align*}
\frac{d}{du} J_{1}(u)
&= -\frac{u^{2H_1-1}}{(T-u)^{H_1-\frac12}}
\int_0^u h_T(s)
\tau(u,s)\,ds
\\
&\quad-
(H_1-\tfrac12)
\int_0^u h_T(s)
\int_u^T(t-u)^{-H_1-\frac12}
\left(u^{2H_1-1}\tau(u,s)
- t^{2H_1-1}\tau(t,s)\right)
\,dt\,ds
\\
&\quad+
(H_1-\tfrac12)
\int_u^T h_T(s)
\int_{s}^T
(t-u)^{-\frac12-H_1}\,t^{2H_1-1} \tau(t,s)
\,dt\,ds.
\end{align*}

Now, let us study $J_2(u)$.
Changing the order of integration, we get
\[
J_2(u)
=
\int_u^T h_T(s)
\int_u^{s}
(t-u)^{\frac12-H_1}\,t^{2H_1-1}\,\psi(t,s)\,dt\,ds
=\int_u^T h_T(s) R(u,s)\,ds,
\]
where
\[
R(u,s) = \int_u^{s}
(t-u)^{\frac12-H_1}\,t^{2H_1-1}\,\psi(t,s)\,dt.
\]
Then 
\[
\frac{d}{du} J_2(u)
=\int_u^T h_T(s) \frac{\partial}{\partial u} R(u,s)\,ds.
\]
Representing $\frac{\partial}{\partial u} R(u,s)$ as a right-sided
Riemann–Liouville derivative on $(0,s)$ and using its Weyl-type representation, we get
\begin{align*}
\frac{\partial}{\partial u}R(u,s)
&= -\Gamma\!\bigl(\tfrac32-H_1\bigr)\,
\mathcal{D}^{\,H_1-\frac12}_{s-}
\Bigl(t^{2H_1-1}\,\psi(t,s)\Bigr)\Big|_{t=u}
\\
&= -\frac{u^{2H_1-1}\,\psi(u,s)}{(s-u)^{H_1-\frac12}}
-(H_1-\tfrac12)\int_u^s
\frac{
u^{2H_1-1}\,\psi(u,s) - t^{2H_1-1}\,\psi(t,s)
}
{(t-u)^{H_1+\frac12}}\,dt.
\end{align*}

Combining the above formulas, we arrive at
\[
A_2(u)
=\frac{1}{\Gamma^2(\frac32-H_1)}\,
\int_0^T h_T(s) K_2(u,s) ds
\]
with
\[
K_2 (u,s) = 
\begin{cases}
K_{2,-} (u,s), & 0 < s < u,\\
K_{2,+} (u,s), & u < s < T,
\end{cases}
\]
where
\begin{align*}
K_{2,-} (u,s) 
&= u^{H_1-\frac12} (T-u)^{\frac12-H_1}
\tau(u,s)
\\
&\quad+(H_1-\tfrac12) u^{\frac12-H_1}\,
\int_u^T(t-u)^{-H_1-\frac12}
\left(u^{2H_1-1}\tau(u,s)
- t^{2H_1-1}\tau(t,s)\right)
\,dt
\end{align*}
and
\begin{align*}
K_{2,+} (u,s)
&= -(H_1-\tfrac12) u^{\frac12-H_1}\,
\int_{s}^T
(t-u)^{-\frac12-H_1}\,t^{2H_1-1} \tau(t,s)
\,dt
\\
&\quad+ 
\frac{u^{H_1-\frac12}\,\psi(u,s)}{(s-u)^{H_1-\frac12}}
+ (H_1-\tfrac12)u^{\frac12-H_1}\int_u^s
\frac{
u^{2H_1-1}\,\psi(u,s) - t^{2H_1-1}\,\psi(t,s)
}
{(t-u)^{H_1+\frac12}}\,dt.
\end{align*}

Combining the representations for $A_1(u)$ and $A_2(u)$, obtained in Steps 2 and~3, and then using the relation $\varphi(u,s) + \tau(u,s) = \rho(u,s)$, we complete the proof. 
\end{proof}

\section{Representation for kernel via hypergeometric functions. Analysis of singularities}
\label{sec:K-hypergeo}

\subsection{Representation for \texorpdfstring{$\psi(u,s)$}{psi}}

In what follows, $\hyper(a, b; c; z)$ denotes the hypergeometric function, see Appendix~\ref{app:hypergeo} for its definition and selected properties.
$C^{\alpha}([0,1])$ is the H\"older space on $[0,1]$ with exponent $\alpha\in(0,1)$.

\begin{lemma}\label{lem:psi-factorized}
Let $\frac12 < H_1 < H_2 < 1$, and let $0<u<s<T$.
Then $\psi(u,s)$ admits the representation
\begin{equation}\label{eq:psi-factor}
\psi(u,s) =
D_1\,u^{1-2H_1}s^{2H_2-2}\,
F_1(\tfrac{u}{s})
+D_2\,u^{2-2H_1}s^{H_1-\frac32}(s-u)^{2H_2-H_1-\frac32}\,
F_2(\tfrac{u}{s}),
\end{equation}
where, for $z\in[0,1)$,
\begin{align}
F_1(z) &= \hyper\!\left(2-2H_2,\tfrac32-H_1;\,3-2H_1;\,z\right),
\label{eq:F1-def-holder}
\\
F_2(z) &= \hyper\!\left(1+2H_2-2H_1,\tfrac32-H_1;\,4-2H_1;\,z\right),
\label{eq:F2-def-holder}
\end{align}
and
\[
D_1 = 2(1-H_1)\,B\!\left(\tfrac32-H_1,\tfrac32-H_1\right),
\qquad
D_2 = (1-H_2)
B\!\left(\tfrac32-H_1,\tfrac32-H_1\right).
\]
Moreover, the functions $F_1$ and $F_2$ admit continuous extensions to $[0,1]$ and satisfy
\[
F_1\in C^{\alpha}([0,1]),
\qquad
F_2\in C^{\beta}([0,1])
\]
with
\begin{equation}\label{eq:alpha-beta}
\alpha:=2H_2-H_1-\tfrac12\in(0,1),
\quad
\beta:=\tfrac32-2H_2+H_1=1-\alpha\in(0,1).
\end{equation}
\end{lemma}

\begin{proof}
\emph{Hypergeometric representation of $\psi(u,s)$.}
Recall the notation $\psi(u,s)=\partial_u G_2(u,s)$, where
$G_2(u,s)$ is given by \eqref{eq:G2}.
With the substitution $t=ux$, we obtain
\[
G_2(u,s)
=
u^{2-2H_1}s^{2H_2-2}
\int_0^1 x^{\frac12-H_1}(1-x)^{\frac12-H_1}
\left(1-\tfrac{u}{s}x\right)^{2H_2-2}\!dx.
\]
By Euler's integral representation~\eqref{eq:hypergeom},
\[
G_2(u,s) = u^{2-2H_1}s^{2H_2-2}
B\left(\tfrac32-H_1,\tfrac32-H_1\right)
{}_2F_1\left(2-2H_2,\tfrac32-H_1;3-2H_1;\tfrac{u}{s}\right).
\]
Differentiating in $u$ and using~\eqref{eq:hypergeo-deriv-app}, we get
\begin{align}
\psi(u,s)
&=
B\!\left(\tfrac32-H_1,\tfrac32-H_1\right)
s^{2H_2-2}u^{1-2H_1}
\notag\\
&\quad{}\times
\Biggl[
(2-2H_1)\,
\hyper\!\left(2-2H_2,\tfrac32-H_1;\,3-2H_1;\,\tfrac{u}{s}\right)
\notag\\
&\qquad\qquad+
\frac{(2-2H_2)(\tfrac32-H_1)}{3-2H_1}\,
\frac{u}{s}\,
\hyper\left(3-2H_2,\tfrac52-H_1;\,4-2H_1;\,\tfrac{u}{s}\right)
\Biggr].
\label{eq:psi-hypergeo}
\end{align}
By the transformation formula \eqref{eq:AS-15-3-3}, the second hypergeometric function in \eqref{eq:psi-hypergeo} can be written as
\[
\hyper\left(3-2H_2,\tfrac52-H_1;\,4-2H_1;\,\tfrac{u}{s}\right)
= (1-\tfrac{u}{s})^{2H_2-H_1-\frac32}  \hyper\!\left(1+2H_2-2H_1,\tfrac32-H_1;\,4-2H_1;\,\tfrac{u}{s}\right).
\]
Substituting this representation into \eqref{eq:psi-hypergeo}, we obtain~\eqref{eq:psi-factor} with
$F_1$, $F_2$ given by~\eqref{eq:F1-def-holder}--\eqref{eq:F2-def-holder}.

\medskip
\emph{H\"older continuity of $F_1$ and $F_2$.}
Since $F_1$ and $F_2$ are hypergeometric functions, they are analytic on $[0,1)$.
It remains to control their behavior at $z=1$.
For $F_1$, we have
\[
F_1(z)=\hyper(a,b;c;z),\quad
a:=2-2H_2,\; b:=\tfrac32-H_1,\; c:=3-2H_1,
\]
hence
\[
\lambda:=c-a-b=2H_2-H_1-\tfrac12=\alpha\in(0,1).
\]
Therefore, by Appendix~\ref{app:hypergeo} (Case $\lambda>0$), $F_1$ admits a continuous
extension to $[0,1]$ and satisfies $F_1\in C^{\alpha}([0,1])$.

Similarly, for $F_2$ we write
\[
F_2(z)=\hyper(a',b';c';z),\quad
a':=1+2H_2-2H_1,\; b':=\tfrac32-H_1,\; c':=4-2H_1,
\]
so that
\[
\lambda':=c'-a'-b'=\tfrac32-2H_2+H_1=\beta\in(0,1).
\]
Again by Appendix~\ref{app:hypergeo} (Case $\lambda>0$), $F_2$ admits a continuous
extension to $[0,1]$ and satisfies $F_2\in C^{\beta}([0,1])$.

This completes the proof.
\end{proof}

\begin{lemma}\label{lem:Phi-factorized}
Let $\frac12 < H_1 < H_2 < 1$ and $0<u<s<T$.
Define
\begin{equation}\label{eq:I-def}
\Phi(u,s)
\coloneqq
\int_u^s
\frac{u^{2H_1-1}\,\psi(u,s) - t^{2H_1-1}\,\psi(t,s)}{(t-u)^{H_1+\frac12}}\,dt,
\end{equation}
where $\psi$ is given by~\eqref{eq:psi}.

Then the integral~\eqref{eq:I-def} is finite for all $0<u<s<T$
and admits the representation
\begin{equation}\label{eq:I-factor-main}
\Phi(u,s) =
D_1 s^{2H_2-2}(s-u)^{\frac12-H_1}\,
\Phi_1\!\left(\tfrac{u}{s}\right)
+ D_2 s^{H_1-\frac12}(s-u)^{2H_2-2H_1-1}\,
\Phi_2\!\left(\tfrac{u}{s}\right),
\end{equation}
where
\begin{align}
\Phi_1(z)
&\coloneqq
\int_0^1
\frac{F_1(z)-F_1\!\bigl(z+(1-z)y\bigr)}{y^{H_1+\frac12}}\,dy,
\label{eq:Phi1-def}
\\
\Phi_2(z)
&\coloneqq
\int_0^1
\frac{
z\,F_2(z)
-
\bigl(z+(1-z)y\bigr)(1-y)^{2H_2 - H_1 - \frac32}
F_2\bigl(z+(1-z)y\bigr)
}{y^{H_1+\frac12}}\,dy,
\label{eq:Phi2-def}
\end{align}
and $D_1$, $D_2$, $F_1$, $F_2$ are defined in Lemma \ref{lem:psi-factorized}.

Moreover, $\Phi_1$ and $\Phi_2$ are bounded and continuous on $[0,1]$.
\end{lemma}

\begin{proof}
\emph{Proof of representation \eqref{eq:I-factor-main}.}
Using~\eqref{eq:psi-factor}, we get
\begin{align*}
u^{2H_1-1}\psi(u,s) - t^{2H_1-1}\psi(t,s)
&= D_1 s^{2H_2-2}
\left[F_1(\tfrac{u}{s})-F_1(\tfrac{t}{s})\right]
\\
&\quad+ D_2 s^{H_1-\frac32}
\left[
u(s-u)^{2H_2 - H_1 - \frac32}F_2(\tfrac{u}{s})
-
t(s-t)^{2H_2 - H_1 - \frac32}
F_2(\tfrac{t}{s})
\right].
\end{align*}
Therefore
\begin{align*}
\Phi(u,s) &= 
s^{2H_2-2}
D_1\int_u^s
\frac{F_1(\frac{u}{s})-F_1(\frac{t}{s})}{(t-u)^{H_1+\frac12}}\,dt
\\
&\quad+
D_2 s^{H_1-\frac32}
\int_u^s
\frac{u(s-u)^{2H_2 - H_1 - \frac32}F_2(\frac{u}{s})
-
t(s-t)^{2H_2- H_1-\frac32}F_2(\frac{t}{s})}{(t-u)^{H_1+\frac12}}\,dt.
\end{align*}
Using the substitution $t = u + (s-u)y$ in both integrals, we arrive at
\eqref{eq:I-factor-main} with $\Phi_1,\Phi_2$ given by
\eqref{eq:Phi1-def}--\eqref{eq:Phi2-def}.

\emph{Boundedness and continuity of $\Phi_1$.}
Recall that $F_1\in C^{\alpha}([0,1])$ with
$\alpha=2H_2-H_1-\tfrac12\in(0,1)$
(Lemma~\ref{lem:psi-factorized}).
Set $z_y:=z+(1-z)y$.
By H\"older continuity,
\[
|F_1(z)-F_1(z_y)|
\le C|z-z_y|^{\alpha}
= C(1-z)^{\alpha}y^{\alpha}
\le C y^{\alpha},
\quad z \in [0,1].
\]
Hence
\[
\left|
\frac{F_1(z)-F_1(z_y)}{y^{H_1+\frac12}}
\right|
\le
C y^{\alpha-(H_1+\frac12)}.
\]
Since $\alpha-(H_1+\frac12)=2H_2-2H_1-1>-1$, the right-hand side is integrable on $(0,1)$.
Thus $\Phi_1(z)$ is well-defined for all $z\in[0,1]$ and uniformly bounded on $[0,1]$.
Finally, continuity of $\Phi_1$ on $[0,1]$ follows from dominated convergence, using the above integrable bound.

\emph{Boundedness and continuity of $\Phi_2$.}
With $z_y:=z+(1-z)y$, we can rewrite~\eqref{eq:Phi2-def} as
\[
\Phi_2(z)
=
\int_0^1
\frac{z\,F_2(z)-z_y(1-y)^{\alpha-1}F_2(z_y)}{y^{H_1+\frac12}}\,dy,
\qquad z\in[0,1].
\]
Adding and subtracting $zF_2(z_y)$ and $z_yF_2(z_y)$ yields
\[
|\Phi_2(z)|
\le
\int_0^1 \frac{|z|\;|F_2(z)-F_2(z_y)|}{y^{H_1+\frac12}}\,dy
+\int_0^1 \frac{|z-z_y|\;|F_2(z_y)|}{y^{H_1+\frac12}}\,dy
+\int_0^1 \frac{|z_y|\;|1-(1-y)^{\alpha-1}|\;|F_2(z_y)|}{y^{H_1+\frac12}}\,dy .
\]
Since $z,z_y\in[0,1]$ and $F_2$ is bounded on $[0,1]$, i.e.,
$\sup_{x\in[0,1]}|F_2(x)|\le C$, we obtain
\[
|\Phi_2(z)|
\le
\int_0^1 \frac{|F_2(z)-F_2(z_y)|}{y^{H_1+\frac12}}\,dy
+ C\int_0^1 \frac{|z-z_y|}{y^{H_1+\frac12}}\,dy
+ C\int_0^1 \frac{|1-(1-y)^{\alpha-1}|}{y^{H_1+\frac12}}\,dy.
\]

We show that each term on the right-hand side is finite.
By Lemma~\ref{lem:psi-factorized}, $F_2\in C^{\beta}([0,1])$ with
$\beta=1-\alpha\in(0,1)$, hence
\[
|F_2(z)-F_2(z_y)|
\le C|z-z_y|^{\beta}
= C(1-z)^{\beta}y^{\beta}
\le C y^{\beta},
\qquad z\in[0,1].
\]
Therefore,
\[
\frac{|F_2(z)-F_2(z_y)|}{y^{H_1+\frac12}}
\le C y^{\beta-(H_1+\frac12)},
\]
which is integrable on $(0,1)$ since $\beta-(H_1+\frac12)=1-2H_2>-1$.

Next,
\[
\frac{|z-z_y|}{y^{H_1+\frac12}}
=
\frac{(1-z)y}{y^{H_1+\frac12}}
\le y^{\frac12-H_1},
\]
and $y^{\frac12-H_1}$ is integrable on $(0,1)$ because $\frac12-H_1>-1$.

Finally, since $\alpha\in(0,1)$, we have $(1-y)^{\alpha-1}\ge 1$ for $y\in(0,1)$, and
\[
(1-y)^{\alpha-1}-1
=
\frac{1-(1-y)^{1-\alpha}}{(1-y)^{1-\alpha}}
\le
\frac{y^{1-\alpha}}{(1-y)^{1-\alpha}},
\]
where we used the Hölder estimate $1-(1-y)^{1-\alpha}\le y^{1-\alpha}$.
Hence,
\[
\frac{|1-(1-y)^{\alpha-1}|}{y^{H_1+\frac12}}
=
\frac{(1-y)^{\alpha-1}-1}{y^{H_1+\frac12}}
\le
y^{\frac12-\alpha-H_1}(1-y)^{\alpha-1}.
\]
The right-hand side is integrable on $(0,1)$ since
$\frac12-\alpha-H_1=1-2H_2>-1$ and $\alpha-1>-1$.

Thus $\Phi_2(z)$ is finite for all $z\in[0,1]$ and uniformly bounded on $[0,1]$.
Continuity of $\Phi_2$ on $[0,1]$ follows from dominated convergence, using the integrable bound
\[
C y^{\beta-(H_1+\frac12)} + C y^{\frac12-H_1} + C y^{\frac12-\alpha-H_1}(1-y)^{\alpha-1},
\]
which is independent of $z\in[0,1]$.

The lemma is proved.
\end{proof}

\subsection{Representation for \texorpdfstring{$\rho(u,s)$}{rho}}

\begin{lemma}\label{lem:rho-factorized}
Let $\frac12 < H_1 < H_2 < 1$, and let $0<s<u<T$.
Then $\rho(u,s)$ admits the representation
\begin{align}
\rho(u,s) &= 
(u-s)^{2H_2 - H_1 - \frac32}\,u^{-H_1-\frac12}(D_3  u + D_5 s)
F_3(\tfrac{s}{u})
\notag\\
&\quad+D_4 u^{-2H_2}s^{2H_2-H_1-\frac12}(u-s)^{2H_2-H_1-\frac12}
F_4(\tfrac{s}{u}) 
\notag\\
&\quad-D_5 u^{1-2H_2}s^{2H_2-H_1-\frac12}(u-s)^{2H_2-H_1-\frac32}
F_5(\tfrac{s}{u}),
\label{eq:rho-factor}
\end{align}
where, for $z\in(0,1)$,
\begin{align}
F_3(z) &= \hyper\!\left(H_1-\tfrac12, \tfrac32-H_1; 2H_2-H_1+\tfrac12;\,1-z\right),
\label{eq:F3-def-holder}
\\
F_4(z) &= \hyper\!\left(2H_2-2H_1+1,2H_2-1;\,2H_2-H_1+\tfrac32;\,1-z\right),
\label{eq:F4-def-holder}
\\
F_5(z) &= \hyper\!\left(2H_2-2H_1,2H_2-1;\,2H_2-H_1+\tfrac12;\,z\right),
\label{eq:F5-def-holder}
\end{align}
and
\begin{align}
D_3 &= 2(H_2-H_1)B\!\left(\tfrac32-H_1,2H_2-1\right),
\label{eq:D_3}
\\
D_4 &= \frac{(H_1-\tfrac12)(\tfrac32-H_1)}{2H_2-H_1+\tfrac12}
B\!\left(\tfrac32-H_1,2H_2-1\right),
\label{eq:D_4}
\\
D_5 &= (H_1-\tfrac12)
B\!\left(\tfrac32-H_1,2H_2-1\right).
\label{eq:D_5}
\end{align}
Moreover, functions $F_3$, $F_4$, and $F_5$ admit continuous extensions to $[0,1]$ and satisfy
\[
F_3 \in C^{\alpha}([0,1]), 
\quad 
F_4 \in C^{\beta}([0,1]), 
\quad
F_5 \in C^{\beta}([0,1]),  
\]
with $\alpha$ and $\beta$ given by \eqref{eq:alpha-beta}.
\end{lemma}
\begin{proof}
Following the definition of $\rho(u,s)$ \eqref{eq:rho}, it will be sufficient to construct a representation for the functions $\varphi(u,s)$ and $\tau(u,s)$ given by \eqref{eq:phi} and \eqref{eq:tau} respectively.

\emph{Representation for $\varphi(u,s)$.}
Set
\[
I(u,s)
\coloneqq
\int_s^u
(u-t)^{\frac12-H_1}\, t^{\frac12-H_1}\, (t-s)^{2H_2-2}\,dt,
\qquad 0<s<u<T,
\]
so that $\varphi(u,s)=\partial_u I(u,s)$.

With the change of variables $t-s=(u-s)(1-x)$, $x\in[0,1]$, we obtain
\[
I(u,s)
=
(u-s)^{2H_2 - H_1 - \frac12}\,u^{\frac12-H_1}
\int_0^1
x^{\frac12-H_1}(1-x)^{2H_2-2}
\Bigl(1-\bigl(1-\tfrac{s}{u}\bigr)x\Bigr)^{\frac12-H_1}\,dx.
\]
By Euler's integral representation~\eqref{eq:hypergeom},
\[
I(u,s)
=
(u-s)^{2H_2 - H_1 - \frac12}\,u^{\frac12-H_1}
B\!\left(\tfrac32 - H_1,\,2H_2-1\right) 
\hyper\!\left(H_1-\tfrac12,\,\tfrac32-H_1;\,2H_2-H_1+\tfrac12;\,1-\tfrac{s}{u}\right).
\]

Differentiating with respect to $u$ and using~\eqref{eq:hypergeo-deriv-app}, we get
\begin{align*}
&\varphi(u,s)
=
B\!\left(\tfrac32 - H_1,\,2H_2-1\right)
\\*
&\quad\times\Biggl[
\Bigl(
(2H_2 - H_1 - \tfrac12)\,(u-s)^{2H_2 - H_1 - \frac32}\,u^{\frac12-H_1}
-(H_1-\tfrac12)\,(u-s)^{2H_2 - H_1 - \frac12}\,u^{-H_1-\frac12}
\Bigr)\\*
&\qquad\quad\times
\hyper\!\left(H_1-\tfrac12,\,\tfrac32-H_1;\,2H_2-H_1+\tfrac12;\,1-\tfrac{s}{u}\right)
\\
&\qquad\quad
+
(u-s)^{2H_2 - H_1 - \frac12}\,u^{\frac12-H_1}
\frac{(H_1-\tfrac12)(\tfrac32-H_1)}{2H_2 - H_1 + \tfrac12}\,
\frac{s}{u^2}
\hyper\!\left(H_1+\tfrac12,\tfrac52-H_1;2H_2-H_1+\tfrac32;1-\tfrac{s}{u}\right)
\Biggr].
\end{align*}

Factoring out the common term
$(u-s)^{2H_2 - H_1 - \frac32}\,u^{-H_1-\frac12}$
and simplifying the coefficient of the first hypergeometric function,
\[
(2H_2-H_1-\tfrac12)u-(H_1-\tfrac12)(u-s)
=2(H_2-H_1)u+(H_1-\tfrac12)s,
\]
we arrive at
\begin{align}
\varphi(u,s)
&=
(u-s)^{2H_2 - H_1 - \frac32}\,u^{-H_1-\frac12}
\notag\\
&\quad\times\biggl[
\Bigl(D_3 u + D_5 s\Bigr)\,
\hyper\!\left(H_1-\tfrac12,\,\tfrac32-H_1;\,2H_2-H_1+\tfrac12;\,1-\tfrac{s}{u}\right)
\notag\\
&\qquad\qquad
+
D_4\,\frac{s(u-s)}{u}
\hyper\!\left(H_1+\tfrac12,\,\tfrac52-H_1;\,2H_2-H_1+\tfrac32;\,1-\tfrac{s}{u}\right)
\biggr],
\label{eq:varphi-compact}
\end{align}
where the constants are given by \eqref{eq:D_3}--\eqref{eq:D_5}.

Next, by the transformation formula~\eqref{eq:AS-15-3-3},
\begin{multline*}
\hyper\!\left(H_1+\tfrac12,\,\tfrac52-H_1;\,2H_2-H_1+\tfrac32;\,1-\tfrac{s}{u}\right)
\\*
=
\left(\tfrac{s}{u}\right)^{2H_2-H_1-\frac32}
\hyper\!\left(2H_2-2H_1+1,\,2H_2-1;\,2H_2-H_1+\tfrac32;\,1-\tfrac{s}{u}\right).
\end{multline*}
Substituting this into~\eqref{eq:varphi-compact} yields the decomposition
\begin{align}
\varphi(u,s)
&=
(u-s)^{2H_2 - H_1 - \frac32}\,u^{-H_1-\frac12}\,
\bigl(D_3 u + D_5 s\bigr)\,
F_3\!\left(\tfrac{s}{u}\right)
\notag\\
&\quad
+
D_4\,
(u-s)^{2H_2 - H_1 - \frac12}\,u^{-2H_2}\,s^{2H_2-H_1-\frac12}\,
F_4\!\left(\tfrac{s}{u}\right),
\label{eq:varphi-via-F}
\end{align}
where $F_3$ and $F_4$ are given by~\eqref{eq:F3-def-holder}--\eqref{eq:F4-def-holder}.

\emph{Representation for $\tau(u,s)$.}
Differentiating with respect to $u$ in the definition~\eqref{eq:tau} gives
\[
\tau(u,s)
=
-(H_1-\tfrac12)\int_0^{s}
(u-t)^{-\frac12-H_1}\, t^{\frac12-H_1}\,(s-t)^{2H_2-2}\,dt.
\]
With the substitution $t=sx$, $x\in[0,1]$, we obtain
\[
\tau(u,s)
=
-(H_1-\tfrac12)\,s^{2H_2-H_1-\frac12}\,u^{-H_1-\frac12}
\int_0^1
x^{\frac12-H_1}(1-x)^{2H_2-2}\left(1-\tfrac{s}{u}x\right)^{-\frac12-H_1}\,dx.
\]
Applying Euler's integral representation~\eqref{eq:hypergeom} yields
\begin{multline}\label{eq:vartau-compact}
\tau(u,s)
=
-(H_1-\tfrac12)\,B\!\left(\tfrac32-H_1,\,2H_2-1\right)
s^{2H_2-H_1-\frac12}\,u^{-H_1-\frac12}
\\*
{}\times
\hyper\!\left(H_1+\tfrac12,\,\tfrac32-H_1;\,2H_2+\tfrac12-H_1;\,\tfrac{s}{u}\right).
\end{multline}

Next, by the transformation formula~\eqref{eq:AS-15-3-3},
\begin{align*}
&\hyper\!\left(H_1+\tfrac12,\,\tfrac32-H_1;\,2H_2+\tfrac12-H_1;\,\tfrac{s}{u}\right)
\\*
&\qquad=
\left(1-\tfrac{s}{u}\right)^{2H_2-H_1-\frac32}
\hyper\!\left(2H_2-2H_1,\,2H_2-1;\,2H_2-H_1+\tfrac12;\,\tfrac{s}{u}\right).
\end{align*}
Substituting this into~\eqref{eq:vartau-compact}, we obtain
\begin{equation}\label{eq:tau-via-F5}
\tau(u,s) = -D_5\,
s^{2H_2-H_1-\frac12}\,u^{-H_1-\frac12}
\left(1-\tfrac{s}{u}\right)^{2H_2-H_1-\frac32}\,
F_5\!\left(\tfrac{s}{u}\right),
\end{equation}
where $F_5$ and $D_5$ are defined in~\eqref{eq:F5-def-holder} and \eqref{eq:D_5}.

Combining~\eqref{eq:varphi-via-F} and~\eqref{eq:tau-via-F5} yields~\eqref{eq:rho-factor}.

\emph{H\"older continuity of $F_3$, $F_4$, and $F_5$.}
We argue as in Lemma~\ref{lem:psi-factorized}.
Since hypergeometric functions are analytic on $[0,1)$, it remains to control their behavior at the endpoint~$1$ of the argument.

For $F_3$, we have
\[
F_3(z) = \hyper(a,b;c;1-z),\quad
a:=H_1-\tfrac12,\; b:=\tfrac32-H_1,\; c:=2H_2-H_1+\tfrac12,
\]
hence
\[
\lambda:=c-a-b = 2H_2-H_1-\tfrac12 = \alpha \in(0,1).
\]
Therefore, by Appendix~\ref{app:hypergeo}, $F_3$ admits a continuous extension to $[0,1]$
and satisfies $F_3\in C^{\alpha}([0,1])$.

Similarly, for $F_4$ we write
\[
F_4(z) = \hyper(a',b';c';1-z),\:
a'=2H_2-2H_1+1,\, b'=2H_2-1,\, c'=2H_2-H_1+\tfrac32,
\]
so that
\[
\lambda':=c'-a'-b' = \tfrac32-2H_2+H_1 = \beta \in(0,1).
\]
Again, $F_4$ admits a continuous extension to $[0,1]$
and $F_4\in C^{\beta}([0,1])$.

Finally, for $F_5$ we have
\[
F_5(z) = \hyper(a'',b'';c'';z),\;
a'':=2H_2-2H_1,\, b'':=2H_2-1,\, c'':=2H_2-H_1+\tfrac12,
\]
and thus
\[
\lambda'':=c''-a''-b'' = \tfrac32-2H_2+H_1 = \beta \in(0,1).
\]
Hence, $F_5$ admits a continuous extension to $[0,1]$
and satisfies $F_5\in C^{\beta}([0,1])$.
This completes the proof.
\end{proof}

\begin{lemma}\label{lem:g-factorized}
Let $\frac12<H_1<H_2<1$ and $0<s<u<T$.
Define
\[
\Psi(u,s)\coloneqq \int_u^T
\frac{u^{2H_1-1}\rho(u,s)-t^{2H_1-1}\rho(t,s)}{(t-u)^{H_1+\frac12}}\,dt.
\]
Then $\Psi(u,s)$ admits the representation
\begin{equation}\label{eq:rho-int-factor}
\Psi(u,s)
=
(u-s)^{2H_2-2H_1-1}\,u^{H_1-\frac12}
\Bigl[
\Psi_1\left(\tfrac{s}{u},\tfrac{T-u}{u-s}\right)+\Psi_2\left(\tfrac{s}{u},\tfrac{T-u}{u-s}\right)+\Psi_3\left(\tfrac{s}{u},\tfrac{T-u}{u-s}\right)
\Bigr],
\end{equation}
where, for $z \in (0,1)$, $R > 0$,
\begin{align*}
\Psi_1(z,R)
&\coloneqq
\int_0^R
x^{-H_1-\frac12}\Bigl(
(D_3+D_5 z)\,F_3(z)
\\
&\qquad\qquad\qquad\qquad-(1+x)^{2H_2-H_1-\frac32}\,
\bigl(1+(1-z)x\bigr)^{H_1-\frac12}\,
\bigl(D_3+D_5 z_x\bigr)\,F_3(z_x)
\Bigr)\,dx,
\\[2mm]
\Psi_2(z,R)
&\coloneqq
D_4\,z^{2H_2-H_1-\frac12} (1-z)
\int_0^R
\frac{F_4(z) - (1+x)^{2H_2-H_1-\frac12}
\bigl(1+(1-z)x\bigr)^{H_1-2H_2-1}
F_4(z_x)
}{x^{H_1+\frac12}}\,dx,
\\[2mm]
\Psi_3(z,R)
&\coloneqq
-D_5\,z^{2H_2-H_1-\frac12}
\int_0^R
\frac{F_5(z)
-(1+x)^{2H_2-H_1-\frac32}\bigl(1+(1-z)x\bigr)^{2H_1-2H_2}
F_5(z_x)}{x^{H_1+\frac12}}\,dx,
\end{align*}
with $z_x\coloneqq \frac{z}{1+(1-z)x}$.
Moreover,
\[
\Psi_1,\Psi_2,\Psi_3\in C\bigl((0,1)\times(0,\infty)\bigr),
\qquad
\sup_{z\in(0,1),\,R>0}|\Psi_j(z,R)|<\infty,\; j=1,2,3.
\]
\end{lemma}

\begin{proof}
\textbf{Representation.}
Decompose $\rho=\rho_1+\rho_2+\rho_3$ as in Lemma~\ref{lem:rho-factorized}.
A substitution $t=u+\delta x$ in each term
\[
\int_u^T\frac{u^{2H_1-1}\rho_j(u,s)-t^{2H_1-1}\rho_j(t,s)}{(t-u)^{H_1+\frac12}}\,dt,
\qquad j=1,2,3,
\]
followed by collecting the common factor
$(u-s)^{2H_2-2H_1-1}u^{H_1-\frac12}$,
yields~\eqref{eq:rho-int-factor} with $\Psi_j$ given in the statement.
We omit routine algebraic details.

\medskip
\textbf{Boundedness.}
Write
\begin{gather*}
\Psi_1(z,R)=\int_0^R \frac{\mathcal R_1(z,x)}{x^{H_1+\frac12}}\,dx,
\qquad
\Psi_2(z,R)=D_4\,z^{2H_2-H_1-\frac12}(1-z)\int_0^R \frac{\mathcal R_2(z,x)}{x^{H_1+\frac12}}\,dx,
\\
\Psi_3(z,R)=-D_5\,z^{2H_2-H_1-\frac12}\int_0^R \frac{\mathcal R_3(z,x)}{x^{H_1+\frac12}}\,dx,
\end{gather*}
where
\begin{gather*}
\mathcal R_1(z,x):=(D_3+D_5z)F_3(z)-Q_1(z,x)\,(D_3+D_5z_x)F_3(z_x),
\\
Q_1(z,x):=(1+x)^{2H_2-H_1-\frac32}\bigl(1+(1-z)x\bigr)^{H_1-\frac12},
\\
\mathcal R_2(z,x):=F_4(z)-Q_2(z,x)\,F_4(z_x),
\quad
Q_2(z,x):=(1+x)^{2H_2-H_1-\frac12}\bigl(1+(1-z)x\bigr)^{H_1-2H_2-1},
\\
\mathcal R_3(z,x):=F_5(z)-Q_3(z,x)\,F_5(z_x),
\quad
Q_3(z,x):=(1+x)^{2H_2-H_1-\frac32}\bigl(1+(1-z)x\bigr)^{2H_1-2H_2}.
\end{gather*}
Recall that $F_3\in C^\alpha([0,1])$ and $F_4,F_5\in C^\beta([0,1])$, hence all $F_k$ are bounded on $[0,1]$.

\smallskip
\emph{(i) Small $x$.}
For $x\in(0,1)$,
\[
|z-z_x|
= \frac{z(1-z)x}{1+(1-z)x}\le x,
\]
and, since each $Q_j(z,\cdot)$ is $C^1$ on $[0,1]$ with
$\sup_{z\in(0,1),\,x\in(0,1)}|\partial_x Q_j(z,x)|<\infty$,
\[
|1-Q_j(z,x)|\le Cx,\qquad j=1,2,3,
\]
uniformly in $z\in(0,1)$.
Using H\"older continuity of $F_3,F_4,F_5$ and the decomposition
\[
F_k(z)-Q_j(z,x)F_k(z_x)=\bigl(F_k(z)-F_k(z_x)\bigr)+\bigl(1-Q_j(z,x)\bigr)F_k(z_x),
\]
we obtain, uniformly in $z\in(0,1)$,
\[
|\mathcal R_1(z,x)|\le C(x^\alpha+x),\quad
|\mathcal R_2(z,x)|\le C(x^\beta+x),\quad
|\mathcal R_3(z,x)|\le C(x^\beta+x),
\]
$x\in(0,1)$.
Therefore the corresponding integrands are dominated on $(0,1)$ by
\[
C\bigl(x^{\alpha-(H_1+\frac12)}+x^{\frac12-H_1}\bigr)
\quad\text{or}\quad
C\bigl(x^{\beta-(H_1+\frac12)}+x^{\frac12-H_1}\bigr),
\]
which are integrable because
$\alpha-(H_1+\frac12)=2H_2-2H_1-1>-1$,
$\beta-(H_1+\frac12)=1-2H_2>-1$, and $\frac12-H_1>-1$.

\smallskip
\emph{(ii) Large $x$.}
For $\Psi_1$ and $\Psi_3$, we start with upper bounds for the power factors $Q_1$ and $Q_3$.
Since\linebreak $1+(1-z)x\le 1+x$ for all $x\ge0$ and $z\in(0,1)$, we have
\[
Q_1(z,x)
=(1+x)^{2H_2-H_1-\frac32}\bigl(1+(1-z)x\bigr)^{H_1-\frac12}
\le (1+x)^{2H_2-2},
\qquad x\ge0,
\]
because $H_1-\frac12>0$.
Moreover, since $2H_1-2H_2<0$ and $1+(1-z)x\ge 1$, we obtain
\[
Q_3(z,x)
=(1+x)^{2H_2-H_1-\frac32}\bigl(1+(1-z)x\bigr)^{2H_1-2H_2}
\le (1+x)^{2H_2-H_1-\frac32},
\quad x\ge0.
\]

Using boundedness of $F_3$ and $F_5$ on $[0,1]$, it follows that for $x\ge 1$,
\[
\frac{|\mathcal R_1(z,x)|}{x^{H_1+\frac12}}
\le
C\left(\frac{1}{x^{H_1+\frac12}}+\frac{Q_1(z,x)}{x^{H_1+\frac12}}\right)
\le
C\left(x^{-H_1-\frac12}+x^{2H_2-H_1-\frac52}\right),
\]
and similarly
\[
\frac{|\mathcal R_3(z,x)|}{x^{H_1+\frac12}}
\le
C\left(\frac{1}{x^{H_1+\frac12}}+\frac{Q_3(z,x)}{x^{H_1+\frac12}}\right)
\le
C\left(x^{-H_1-\frac12}+x^{2H_2-2H_1-2}\right).
\]
Both right-hand sides are integrable on $(1,\infty)$ since $-H_1-\frac12<-1$, $2H_2-H_1-\frac52<-1$ and $2H_2-2H_1-2<-1$.

For $\Psi_2$, we split $[1,\infty)$ into $[1,(1-z)^{-1}]$ and $[(1-z)^{-1},\infty)$.
If $1\le x\le (1-z)^{-1}$, then $(1-z)x\le1$ and thus
$(1+(1-z)x)^{H_1-2H_2-1}\le 1$; hence, using boundedness of $F_4$,
\[
(1-z)\frac{|\mathcal R_2(z,x)|}{x^{H_1+\frac12}}
\le C(1-z)\bigl(x^{-H_1-\frac12}+x^{2H_2-2H_1-1}\bigr),
\]
whose integral over $[1,(1-z)^{-1}]$ is uniformly bounded.
If $x\ge (1-z)^{-1}$, then $1+(1-z)x\ge(1-z)x$ and, because $H_1-2H_2-1<0$,
\[
Q_2(z,x)\le C(1-z)^{H_1-2H_2-1}x^{-\frac32},
\]
so
\[
(1-z)\frac{|\mathcal R_2(z,x)|}{x^{H_1+\frac12}}
\le C\Bigl((1-z)x^{-H_1-\frac12}+(1-z)^{H_1-2H_2}x^{-H_1-2}\Bigr),
\]
which is integrable on $[(1-z)^{-1},\infty)$ with a bound independent of $z$.

Thus we obtain an $L_1(0,\infty)$-majorant for each integrand, uniform in $z$
(and with $(1-z)$ included for $\Psi_2$), which implies
\[
\sup_{z\in(0,1),\,R>0}|\Psi_j(z,R)|<\infty,\qquad j=1,2,3.
\]

\smallskip
\textbf{Continuity.}
Fix $(z,R)\in(0,1)\times(0,\infty)$ and let $(z_n,R_n)\to(z,R)$.
For each $x>0$ the integrands converge pointwise by continuity of $z\mapsto z_x$, the power factors, and $F_3,F_4,F_5$.
Since $R_n\le R+1$ for $n$ large and $z_n$ stays in a compact subset of $(0,1)$, the majorants from $(i)$--$(ii)$ are integrable on $[0,R+1]$ and dominate the truncated integrands.
Dominated convergence yields $\Psi_j(z_n,R_n)\to\Psi_j(z,R)$, hence
$\Psi_1,\Psi_2,\Psi_3\in C((0,1)\times(0,\infty))$.
\end{proof}

\begin{remark}\label{rem:Psi2-vanishes}
Inspecting the proof of the boundedness of $\Psi_2$, one actually obtains a decay rate as $z\uparrow1$:
there exists $C>0$ such that
\[
\sup_{R>0}|\Psi_2(z,R)|
\le
C(1-z)^{1+2H_1-2H_2},
\qquad z\in(0,1).
\]
This follows by integrating the majorants obtained in the proof in the regions
$x\in(0,1)$, $1\le x\le (1-z)^{-1}$, and $x\ge (1-z)^{-1}$.
Consequently,
\[
\lim_{z\uparrow1}\sup_{R>0}|\Psi_2(z,R)|=0,
\]
because $1+2H_1-2H_2>0$, so $\Psi_2$ admits a continuous extension at $z=1$, given by $\Psi_2(1,R)=0$ for all $R>0$.
\end{remark}

\begin{lemma}\label{lem:q-factorized}
Let $\frac12<H_1<H_2<1$ and $0<u<s<T$. Define
\[
\Lambda(u,s)\coloneqq \int_{s}^T (t-u)^{-\frac12-H_1}\,t^{2H_1-1}\rho(t,s)\,dt.
\]
Then $\Lambda(u,s)$ admits the representation
\begin{equation}\label{eq:q-factor}
\Lambda(u,s)
=
(s-u)^{2H_2-2H_1-1}\,s^{H_1-\frac12}
\Bigl[
\Lambda_1(\tfrac{u}{s},\tfrac{T-s}{s-u})+\Lambda_2(\tfrac{u}{s},\tfrac{T-s}{s-u})+\Lambda_3(\tfrac{u}{s},\tfrac{T-s}{s-u})
\Bigr],
\end{equation}
where, for $z \in (0,1)$ and $R>0$,
\begin{align*}
\Lambda_1(z,R)
&\coloneqq
\int_0^R
x^{2H_2-H_1-\frac32}(1+x)^{-H_1-\frac12}
\bigl(1+(1-z)x\bigr)^{H_1-\frac32}
\Bigl[(D_3+D_5)+D_3(1-z)x\Bigr]\,
F_3(y_x)\,dx,
\\
\Lambda_2(z,R)
&\coloneqq
D_4(1-z)
\int_0^R
x^{2H_2-H_1-\frac12}(1+x)^{-H_1-\frac12}
\bigl(1+(1-z)x\bigr)^{2H_1-2H_2-1}\,
F_4(y_x)\,dx,
\\
\Lambda_3(z,R)
&\coloneqq
-D_5\!\int_0^R\!
x^{2H_2-H_1-\frac32}(1+x)^{-H_1-\frac12}
\bigl(1+(1-z)x\bigr)^{2H_1-2H_2}
F_5(y_x)dx,
\end{align*}
with
\[
y_x\coloneqq \frac{1}{1+(1-z)x}.
\]
Moreover,
\[
\Lambda_1,\Lambda_2,\Lambda_3\in C\bigl((0,1)\times(0,\infty)\bigr),
\qquad
\sup_{z\in(0,1),\,R>0}|\Lambda_j(z,R)|<\infty,\; j=1,2,3.
\]
\end{lemma}

\begin{proof}[Proof sketch]
The proof follows the same pattern as Lemma~\ref{lem:g-factorized}, but is simpler because the integrals
contain no difference terms.
Insert the decomposition $\rho=\rho_1+\rho_2+\rho_3$ from Lemma~\ref{lem:rho-factorized} into $\Lambda(u,s)$.
Set $\delta=s-u$, $z=u/s\in(0,1)$, $R=(T-s)/\delta$ and use the substitution $t=s+\delta x$, $x\in[0,R]$.
Collecting the common factor $\delta^{2H_2-2H_1-1}s^{H_1-\frac12}$ yields \eqref{eq:q-factor} with the stated
$\Lambda_1,\Lambda_2,\Lambda_3$ (here $y_x=(1+(1-z)x)^{-1}$).

For boundedness and continuity, note that the integrands are continuous in $(z,x)$ and $F_3,F_4,F_5$ are bounded on $[0,1]$.
As in Lemma~\ref{lem:g-factorized}, split the $x$-integrals into
$(0,1)$, $(1,(1-z)^{-1})$, $((1-z)^{-1},\infty)$,
and estimate each part by elementary power bounds.
\end{proof}

\begin{remark}\label{rem:Lambda2-vanishes}
Similarly to Remark \ref{rem:Psi2-vanishes}, one may obtain a decay rate of $\Lambda_2$ as $z\uparrow1$.
More precisely, there exists a constant $C>0$ such that
\[
\sup_{R>0}|\Lambda_2(z,R)|
\le
C(1-z)^{1+2H_1-2H_2},
\qquad z\in(0,1).
\]
Consequently,
$\lim_{z\uparrow1}\sup_{R>0}|\Lambda_2(z,R)|=0$,
because $1+2H_1-2H_2>0$.
In particular, $\Lambda_2$ admits a continuous extension at $z=1$,
given by $\Lambda_2(1,R)=0$ for all $R>0$.
\end{remark}

\subsection{Representation of \texorpdfstring{$K(u,s)$}{K}}
\begin{thm}\label{thm:K-factorized}
Let $\frac12<H_1<H_2<1$ and let $K(u,s)$ be the kernel defined in
Theorem~\ref{thm:fredholm}.  
Then for $u,s \in (0,T)$, $u\neq s$, the kernel admits the representation
\begin{equation}\label{eq:K-factorized}
K(u,s) = |u-s|^{2H_2-2H_1-1} L(u,s),
\end{equation}
where
\[
L(u,s)=
\begin{cases}
L_-(u,s), & 0<s<u<T,\\[2mm]
L_+(u,s), & 0<u<s<T,
\end{cases}
\]
and the functions $L_-$ and $L_+$ are given by
\begin{align}
L_-(u,s)
&=(T-u)^{\frac12-H_1}(u-s)^{H_1-\frac12}
\notag\\
&\quad\times\Bigl[
(D_3+ D_5 \tfrac{s}{u})
F_3\left(\tfrac{s}{u}\right)
+ D_4(\tfrac{s}{u})^{2H_2-H_1-\frac12}
(1-\tfrac{s}{u})
F_4\left(\tfrac{s}{u}\right)
- D_5 (\tfrac{s}{u})^{2H_2-H_1-\frac12}
F_5\left(\tfrac{s}{u}\right)
\Bigl]
\notag\\
&\quad+ (H_1-\tfrac12)
\sum_{j=1}^3\Psi_j(\tfrac{s}{u},\tfrac{T-u}{u-s}),
\label{eq:Lminus-def}
\end{align}
and
\begin{align}
L_+(u,s)
&=
D_1u^{\frac12-H_1}s^{2H_2-2}(s-u)^{\frac32+H_1-2H_2}
\Bigl[
F_1\left(\tfrac{u}{s}\right)
+(H_1-\tfrac12)\Phi_1(\tfrac{u}{s})
\Bigr]
+ D_2\,(\tfrac{u}{s})^{\frac32-H_1}
F_2(\tfrac{u}{s})
\notag\\
&\quad
+ (H_1-\tfrac12)D_2 (\tfrac{u}{s})^{\frac12-H_1}
\Phi_2(\tfrac{u}{s})
-(H_1-\tfrac12)(\tfrac{u}{s})^{\frac12-H_1}
\sum_{j=1}^3
\Lambda_j(\tfrac{u}{s},\tfrac{T-s}{s-u}).
\label{eq:Lplus-def}
\end{align}

Both one-sided limits
\[
\lim_{\substack{(u,s)\to(t,t)\\ 0<s<u<T}}L_-(u,s),
\qquad
\lim_{\substack{(u,s)\to(t,t)\\ 0<u<s<T}}L_+(u,s),
\qquad t\in(0,T),
\]
exist, are independent of $t$, and coincide.
Hence, $L$ admits a continuous extension to the diagonal $\{(t,t):t\in(0,T)\}$.
\end{thm}

\begin{proof}
\emph{Step 1. Representation.}
Insert the representations from
Lemmas~\ref{lem:psi-factorized}, \ref{lem:Phi-factorized},
\ref{lem:rho-factorized}, \ref{lem:g-factorized}, and \ref{lem:q-factorized}
into the formula for $K_\pm$ in Theorem~\ref{thm:fredholm}.
In each region $s<u$ and $u<s$, factor out the common power $|u-s|^{2H_2-2H_1-1}$.
This gives \eqref{eq:K-factorized} with $L_\pm$ defined by
\eqref{eq:Lminus-def}--\eqref{eq:Lplus-def}.

\smallskip
\emph{Step 2. Passage to the diagonal.}
Fix $t\in(0,T)$.
In \eqref{eq:Lminus-def}, the explicit $\rho$-part is multiplied by $(u-s)^{H_1-\frac12}$,
hence it vanishes as $(u,s)\to(t,t)$ with $s<u$ since $H_1-\frac12>0$.
By Lemma~\ref{lem:g-factorized}, $\Psi_j\in C((0,1)\times(0,\infty))$ and are uniformly bounded.
Moreover, by Remark~\ref{rem:Psi2-vanishes}, $\Psi_2(z,R)\to0$ as $z\uparrow1$ uniformly in $R$,
and the same majorants as in Lemma~\ref{lem:g-factorized} yield dominated convergence on $(0,\infty)$.
Therefore the limits $\Psi_j(1,\infty):=\lim_{z\uparrow1,\,R\to\infty}\Psi_j(z,R)$ exist for $j=1,3$, and
\begin{equation}\label{eq:Lminus-diag}
\lim_{\substack{(u,s)\to(t,t)\\ s<u}}L_-(u,s)
=
(H_1-\tfrac12)\bigl[\Psi_1(1,\infty)+\Psi_3(1,\infty)\bigr]
\eqqcolon \ell_- .
\end{equation}

Similarly, in \eqref{eq:Lplus-def} the first line is multiplied by
$(s-u)^{\frac32+H_1-2H_2}$, which tends to $0$ as $(u,s)\to(t,t)$ with $u<s$
because $\frac32+H_1-2H_2>0$.
Furthermore, by Lemma~\ref{lem:q-factorized} and Remark~\ref{rem:Lambda2-vanishes},
the limits $\Lambda_j(1,\infty):=\lim_{z\uparrow1,\,R\to\infty}\Lambda_j(z,R)$ exist and
$\Lambda_2(1,\infty)=0$.
Hence,
\begin{equation}\label{eq:Lplus-diag}
\lim_{\substack{(u,s)\to(t,t)\\ u<s}}L_+(u,s)
=
D_2F_2(1)+(H_1-\tfrac12)D_2\Phi_2(1)
\\*
{}-(H_1-\tfrac12)\bigl[\Lambda_1(1,\infty)+\Lambda_3(1,\infty)\bigr]
\eqqcolon \ell_+ .
\end{equation}

\smallskip
\emph{Step 3. Evaluation of the $\Psi$-constants.}
At $z=1$ we have $z_x\equiv1$. Moreover, $F_3(1)=1$ as the value of a
hypergeometric function at zero. Passing to the limit in the expressions for
$\Psi_1$ and $\Psi_3$ from Lemma~\ref{lem:g-factorized}, we obtain
\begin{align*}
\Psi_1(1,\infty)
&=(D_3+D_5)\int_0^\infty x^{-H_1-\frac12}\Bigl[1-(1+x)^{2H_2-H_1-\frac32}\Bigr]\,dx,
\\
\Psi_3(1,\infty)
&=-D_5F_5(1)\int_0^\infty x^{-H_1-\frac12}\Bigl[1-(1+x)^{2H_2-H_1-\frac32}\Bigr]\,dx .
\end{align*}
By Lemma~\ref{lem:beta-integrals}$(ii)$,
\begin{equation}\label{eq:I-psi}
\int_0^\infty x^{-H_1-\frac12}\Bigl[1-(1+x)^{2H_2-H_1-\frac32}\Bigr]\,dx
=\frac{\frac32+H_1-2H_2}{H_1-\frac12}\,
B \left(\tfrac32-H_1,\,1+2H_1-2H_2\right).
\end{equation}
Combining \eqref{eq:Lminus-diag}--\eqref{eq:I-psi} yields
\begin{equation}\label{eq:Lminus-diag-evaluated}
\ell_-
=
\left(\tfrac32+H_1-2H_2\right)
B\!\left(\tfrac32-H_1,\,1+2H_1-2H_2\right)\,\Xi,
\end{equation}
where
\begin{equation}\label{eq:Xi-def}
\Xi:=(D_3+D_5)-D_5F_5(1).
\end{equation}

\smallskip
\emph{Step 4. Evaluation of the $\Lambda$-constants.}
At $z=1$ we have $y_x\equiv1$, $F_3(1)=1$, and $\Lambda_2(1,\infty)=0$.
Thus, Lemma~\ref{lem:q-factorized} gives
\begin{align*}
\Lambda_1(1,\infty)
&=(D_3+D_5)\int_0^\infty x^{2H_2-H_1-\frac32}(1+x)^{-H_1-\frac12}\,dx,
\\
\Lambda_3(1,\infty)
&=-D_5F_5(1)\int_0^\infty x^{2H_2-H_1-\frac32}(1+x)^{-H_1-\frac12}\,dx .
\end{align*}
The integral equals
\[
\int_0^\infty x^{2H_2-H_1-\frac32}(1+x)^{-H_1-\frac12}\,dx
=
B\!\left(2H_2-H_1-\tfrac12,\,1+2H_1-2H_2\right),
\]
and therefore
\begin{equation}\label{eq:Lam-sum}
(H_1-\tfrac12)\bigl[\Lambda_1(1,\infty)+\Lambda_3(1,\infty)\bigr]
=
(H_1-\tfrac12)B\left(2H_2-H_1-\tfrac12,1+2H_1-2H_2\right)\Xi .
\end{equation}

\smallskip
\emph{Step 5. Evaluation of the $(\Phi_2,F_2)$ contribution.}
By Lemma~\ref{lem:Phi-factorized} and dominated convergence,
\[
\Phi_2(1)
=
F_2(1)\int_0^1\frac{1-(1-y)^{2H_2-H_1-\frac32}}{y^{H_1+\frac12}}\,dy .
\]
Applying Lemma~\ref{lem:beta-integrals}$(i)$ yields
\[
\Phi_2(1)
=
-\frac{F_2(1)}{H_1-\frac12}
\Bigl[(2H_1-2H_2)B\!\left(\tfrac32-H_1,\,2H_2-H_1-\tfrac12\right)+1\Bigr].
\]
Consequently,
\[
A \coloneqq D_2F_2(1)+(H_1-\tfrac12)D_2\Phi_2(1)
= 2(H_2-H_1)D_2F_2(1)
B(\tfrac32-H_1,2H_2-H_1-\tfrac12).
\]
Using the explicit definitions of $D_2$ and $F_2(1)$, together with the value
of $F_2(1)$ from \eqref{eq:hypergeo-limit-1}, we obtain
\begin{equation}\label{eq:first-two-terms-closed}
A =
\frac{\Gamma^2(\frac32-H_1)\Gamma(\frac32+H_1-2H_2)\Gamma(2H_2-H_1-\frac12)}
{\Gamma(2-2H_2)\Gamma(2H_2-2H_1)} .
\end{equation}

\smallskip
\emph{Step 6: Coincidence of the two diagonal limits.}
It remains to prove that $\ell_+=\ell_-$.
Set
\begin{equation}\label{eq:AB}
\begin{split}
B &:= (H_1-\tfrac12)B\!\left(2H_2-H_1-\tfrac12,\,1+2H_1-2H_2\right),
\\
C &:= \left(\tfrac32+H_1-2H_2\right)
B\!\left(\tfrac32-H_1,\,1+2H_1-2H_2\right).
\end{split}
\end{equation}
In this notation,
\[
\ell_+ = A - B \Xi,\qquad
\ell_- = C \Xi,
\]
where the first identity follows from \eqref{eq:Lplus-diag},
\eqref{eq:first-two-terms-closed}, and \eqref{eq:Lam-sum}, while the second one
is \eqref{eq:Lminus-diag-evaluated}. Hence $\ell_+=\ell_-$ is equivalent to
\begin{equation}\label{eq:A-BCXi}
A=(B+C)\Xi .
\end{equation}

To verify \eqref{eq:A-BCXi}, rewrite \eqref{eq:Xi-def} as
\begin{equation}\label{eq:Xi-final}
\Xi=\Gamma(2H_2-1)\left[
\frac{\Gamma(\frac32-H_1)}{\Gamma(2H_2-H_1-\frac12)}
-\frac{\Gamma(\frac32-2H_2+H_1)}{\Gamma(H_1-\frac12)}
\right],
\end{equation}
which follows by inserting the explicit expressions for $D_3$, $D_5$, and the
value of $F_5(1)$ computed by applying \eqref{eq:hypergeo-limit-1}.

Using $B(x,y)=\Gamma(x)\Gamma(y)/\Gamma(x+y)$ and
$\Gamma(H_1+\tfrac12)=(H_1-\tfrac12)\Gamma(H_1-\tfrac12)$, we get from \eqref{eq:AB} that
\begin{equation}\label{eq:B-gamma-noab}
B=\frac{\Gamma(2H_2-H_1-\frac12)\Gamma(1+2H_1-2H_2)}{\Gamma(H_1-\frac12)},
\qquad
C=\frac{\Gamma(\frac32-H_1)\Gamma(1+2H_1-2H_2)}{\Gamma(\frac32-2H_2+H_1)}.
\end{equation}
Combining \eqref{eq:B-gamma-noab} with the expression for $C$ and \eqref{eq:Xi-final}, we obtain
\begin{align}
(B+C)\Xi
&=\Gamma(2H_2-1)\Gamma(1+2H_1-2H_2)
\notag\\
&\quad\times
\biggl[
\frac{\Gamma^2(\frac32-H_1)}
{\Gamma(2H_2-H_1-\frac12)\Gamma(\frac32-2H_2+H_1)}
-\frac{\Gamma(2H_2-H_1-\frac12)\Gamma(\frac32-2H_2+H_1)}
{\Gamma^2(H_1-\frac12)}
\biggr],
\label{eq:BCXi-expanded-noab}
\end{align}
since the mixed terms cancel.

Finally, applying Euler's reflection formula $\Gamma(x)\Gamma(1-x)=\frac{\pi}{\sin(\pi x)}$ to the pairs
$\Gamma\!\left(\tfrac32-H_1\right)\Gamma\!\left(H_1-\tfrac12\right)$
and
$\Gamma\!\left(2H_2-H_1-\tfrac12\right)\Gamma\!\left(\tfrac32-2H_2+H_1\right)$, and using 
$\sin^2 x - \sin^2 y = \sin(x-y) \sin(x+y)$,
we can rewrite \eqref{eq:BCXi-expanded-noab} as
\begin{multline}\label{eq:BCXi-sine-noab}
(B+C)\Xi
=
\Gamma(2H_2-1)\Gamma(1+2H_1-2H_2)
\frac{\pi}{\Gamma^2(H_1-\frac12)}\,
\frac{\sin\bigl(\pi(2H_2-2H_1)\bigr)\sin\bigl(\pi(2H_2-1)\bigr)}
{\sin^2\!\bigl(\pi(H_1-\frac12)\bigr)\sin\bigl(\pi(2H_2-H_1-\frac12)\bigr)}.
\end{multline}

On the other hand, starting from \eqref{eq:first-two-terms-closed} and using
the reflection identities
\begin{gather*}
\frac{1}{\Gamma(2-2H_2)}
=\frac{\Gamma(2H_2-1)\sin(\pi(2H_2-1))}{\pi},
\\
\frac{1}{\Gamma(2H_2-2H_1)}
=\frac{\Gamma(1+2H_1-2H_2)\sin(\pi(2H_2-2H_1))}{\pi},
\\
\Gamma\!\left(\tfrac32+H_1-2H_2\right)\Gamma\!\left(2H_2-H_1-\tfrac12\right)
=\frac{\pi}{\sin\bigl(\pi(2H_2-H_1-\tfrac12)\bigr)},
\shortintertext{and}
\Gamma^2\!\left(\tfrac32-H_1\right)
=\frac{\pi^2}{\Gamma^2(H_1-\tfrac12)\sin^2(\pi(H_1-\tfrac12))},
\end{gather*}
we obtain exactly the same expression as in \eqref{eq:BCXi-sine-noab}.
Therefore $A=(B+C)\Xi$, hence $\ell_+=\ell_-$, and $L$ admits a continuous extension to the diagonal.
\end{proof}

\begin{remark}
Both the kernel $L(u,s)$ and the right-hand side $g_T(u)$ of the Fredholm equation \eqref{Fred_mle} are unbounded as $u\downarrow 0$ or $u\uparrow T$, due to the presence of the factors $u^{\frac12-H_1}$ and $(T-u)^{\frac12-H_1}$. 
In view of this, it is convenient to transform the equation as follows.
Define
\[
\widetilde h_T(u) = h_T(u)\, u^{H_1-\frac12}(T-u)^{H_1-\frac12},
\quad
\widetilde L(u,s) = L(u,s)\,
\frac{u^{H_1-\frac12}(T-u)^{H_1-\frac12}}
     {s^{H_1-\frac12}(T-s)^{H_1-\frac12}}.
\]
Then equation \eqref{Fred_mle} takes the form
\[
\widetilde h_T(u)
+ c(H_1,H_2)\! \int_0^{T}
\frac{\widetilde L(u,s)}{|u-s|^{1+2H_1-2H_2}}
\,\widetilde h_T(s)\,ds
=
\frac{1}
{2H_1 B \left(\frac{3}{2}-H_1,H_1+\frac{1}{2}\right)}
\]
for almost all $u\in(0,T)$. In particular, the right-hand side becomes constant, which simplifies the numerical solution of the equation.

However, the transformed kernel $\widetilde L(u,s)$ remains unbounded on the boundary of $[0,T]^2$: in this case it diverges as $s\downarrow 0$ and as $s\uparrow T$. In other words, the kernel still contains so-called ``additional singularities'' in the sense of \cite{Makogin21,MRZ2024}, which must be taken into account in the numerical treatment. To handle this, we employ the numerical approximation methods developed in \cite{MRZ2024} specifically for kernels of this type. 

We note that the kernel considered in \cite{MRZ2024} was unbounded on different parts of the boundary of $[0,T]^2$, namely along the coordinate axes; however, the adaptation of the method to the present setting is straightforward. 

Figures~\ref{fig:kernels} illustrate the surface plots of the kernels $K(u,s)$ and $L(u,s)$. While $K(u,s)$ exhibits a singularity along the diagonal $u=s$, the kernel $L(u,s)$ remains bounded and continuous there. At the same time, both kernels display divergent behaviour as $u\downarrow 0$ or $u\uparrow T$.
\end{remark}

\begin{figure}
    \centering
    \subfigure[$K(u,s)$]{\includegraphics[width=0.49\linewidth]{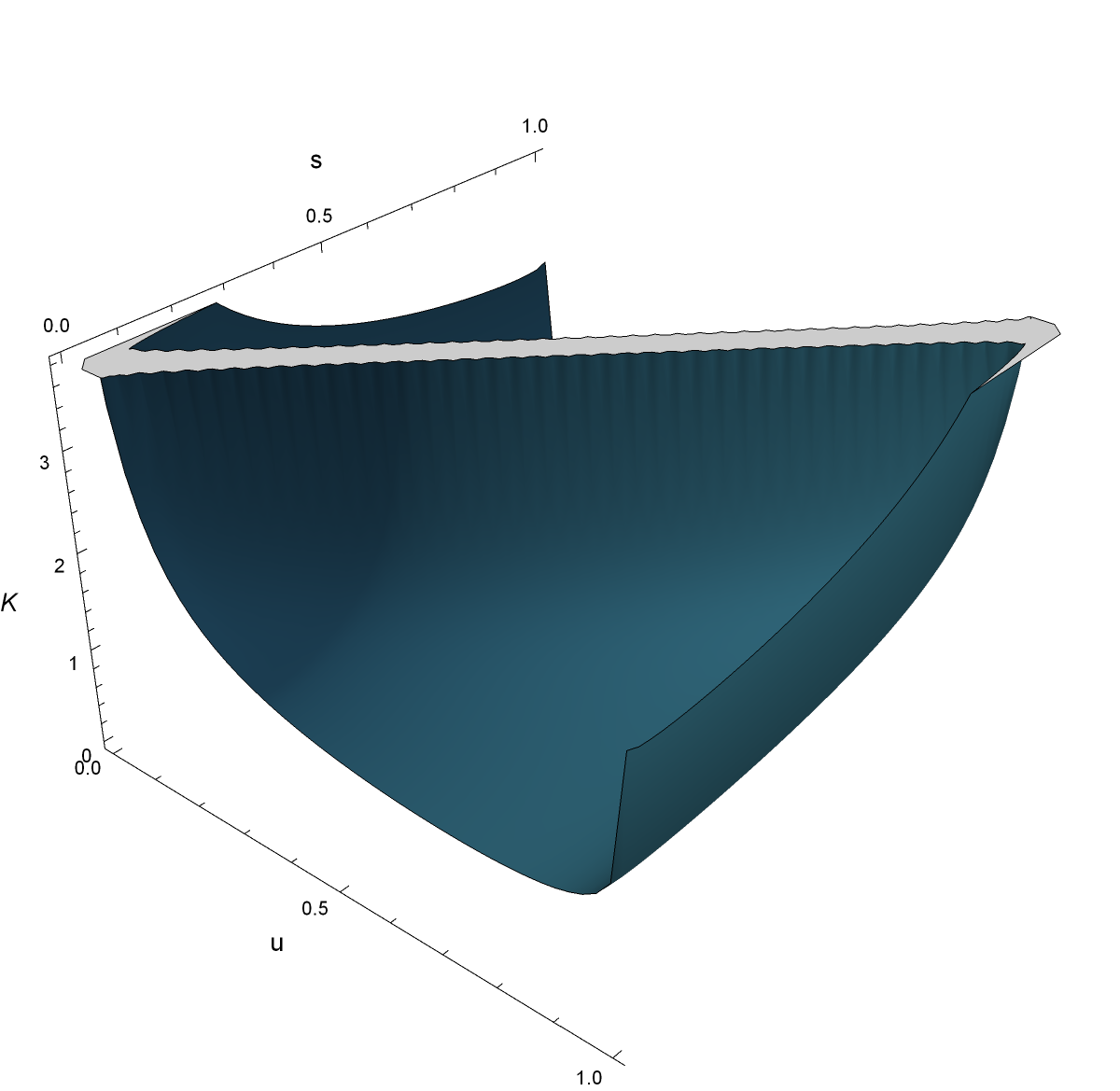}}
    \subfigure[$L(u,s)$]{\includegraphics[width=0.49\linewidth]{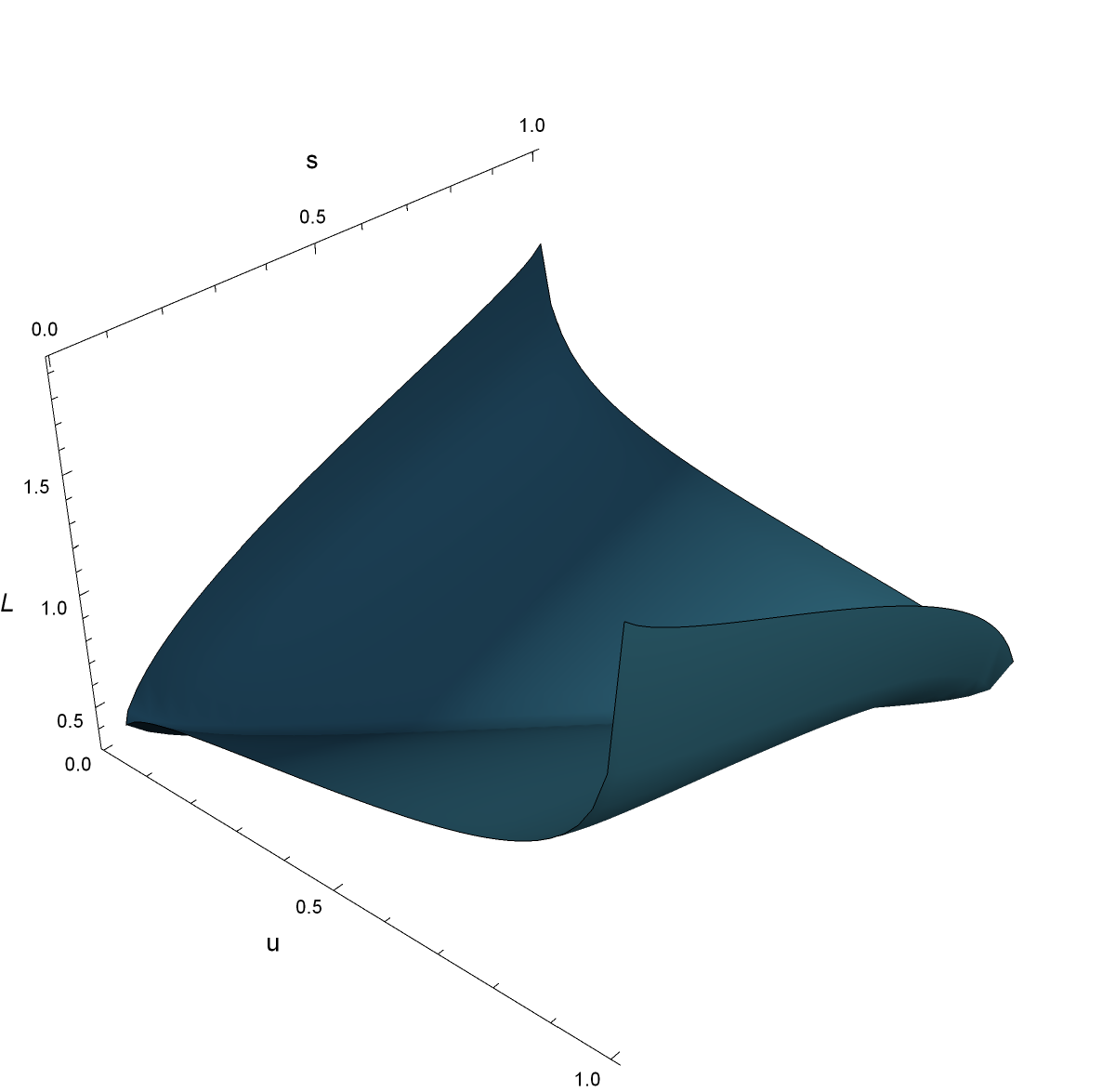}}
    \caption{Surface plots of the kernel $K(u,s)$ and its numerator $L(u,s)$ for $H_1 = 0.7$, $H_2 = 0.85$, and $T=1$.}
    \label{fig:kernels}
\end{figure}

\section{Numerical simulations}
\label{sec:simul}

This section describes the numerical solution of the Fredholm equation arising in the construction of the MLE and presents simulation results illustrating the accuracy of the proposed method.

\subsection{Numerical solution of the Fredholm equation}

To approximate the solution of the Fredholm equation \eqref{Fred_mle}, we employ a modified product-integration method for weakly singular kernels (see, e.\,g., \cite{Prem}). This method is based on a modification proposed by Neta \cite{Neta}. Here we describe the main steps of this numerical method, adapted for our integral equation.

We consider the following generalization of the equation~\eqref{Fred_mle}:
\begin{equation}\label{Fred_mle1}
h_T(u) + c(H_1,H_2)\, \int_0^{T} h_T(s)\, |u-s|^{\gamma-1} L(u,s)\,ds = g(u),
\quad u \in [0,T],
\end{equation}
where the right-hand side $g\in L_2[0,T]$ is an arbitrary given function and $\gamma = 2H_2-2H_1 >0$. 
Define
\[
L^{(n)}(s,u) = L\left((s\vee\tfrac{1}{n})\wedge(T-\tfrac{1}{n}),
        (u\vee\tfrac{1}{n})\wedge(T-\tfrac{1}{n})\right).
\]
Then, obviously, $L^{(n)}(s,u)=L (s ,u ) $ for $s,u \in [1/n, T-1/n]$.
The corresponding approximate equation has the form
\begin{equation}\label{equout1}
g(u) =    h^{(n)}_T(u) + c(H_1,H_2)\int_0^T  h^{(n)}_T(s) |u-s|^{\gamma-1} L^{(n)}(u,s)\,ds,
\end{equation}

We start with a given number $N+1 \geq 1$ of equally spaced points
$0= t_1 < t_2 < \dots < t_{N+1} =T$ and $\delta= t_{j+1}-t_j=\frac{T}{N}$ (i.\,e.\ $t_j=(j-1)\delta$, $j=1,\dots,N+1$).  By the product-integration rule
\[
\int_{t_j}^{t_{j+1}}L^{(n)}(s,u)|u-s|^{\gamma-1} h^{(n)}_T(s)ds
\approx \int_{t_j}^{t_{j+1}} \frac{L_j(u)(t_{j+1}-s)h_j + (s-t_j)L_{j+1}(u)h_{j+1}}{\delta}|u-s|^{\gamma-1}ds,
\]
where $L_j(u)=L^{(n)}(t_j,u)$ and $h_j = h^{(n)}_T(t_j)$. Then we have
\begin{equation}\label{f_sum}
g(u)=c(H_1,H_2)\sum_{j=1}^{N-1} \left(L_j(u) h_j \psi_{j}^{1}(u)+ L_{j+1}(u) h_{j+1} \psi_ {j+1}^{2} (u)\right)+h^{(n)}_T(u),
\end{equation}
where the weights are defined by
\[
\psi_{j}^{1} (u) = \frac{1}{\delta} \int_{t_j}^{t_{j+1}} (t_{j+1}-s) |u-s|^{\gamma-1} ds,
\qquad
\psi_{j}^{2} (u) = \frac{1}{\delta} \int_{t_{j-1}}^{t_j} (s-t_{j-1}) |u-s|^{\gamma-1} ds.
\]
Substituting $u=t_i$ in \eqref{f_sum} and rearranging the sums, we obtain the following system of linear equations for the approximation of the integral equation \eqref{equout1}
\begin{equation}\label{fj_sum}
h_i+c(H_1,H_2)\sum_{j=1}^{N+1} L_{i,j} \left( \psi_{j, i}^{1}+ \psi_{j, i}^{2} \right) h_j  =g_i,
\quad i = 1, \ldots, N+1,
\end{equation}
where $L_{i,j}=L^{(n)}(t_j,t_i),$ $\psi_{j, i} ^{l} = \psi_{j}^l (t_i)$ for $l =1,2,$ and $g_i= g(t_i)$.
In \eqref{fj_sum}, we assume that $\psi_{N, i}^{1} = \psi_{0, i}^{2} = 0$ for all $i$.
From a practical point of view, it is convenient to choose $n = N$. Then
\[
L_{i,j}=
\begin{cases}
L(t_j,t_i), & i=2,\dots, N, \; j=2,\dots, N,\\
L(1/N,t_i), & j=1,\\
L(t_j,1/N), & i=1,\\
L(T-1/N,t_i), & j=N+1,\\
L(t_j,T-1/N), & i=N+1.
\end{cases}
\]
On the selected grid, weights then can be calculated explicitly (and independent of function $L(u,s)$): 
\begin{gather*}
    \psi_{j, i}^{1}=
    \begin{cases}
    \frac{\delta^\gamma}{\gamma(\gamma+1)}\bigl((j-i+1)^{\gamma+1}-(j-i)^{\gamma}(j-i+\gamma+1)\bigr), & i\le j,\\
    \frac{\delta^\gamma}{\gamma(\gamma+1)}\bigl((i-j-1)^{\gamma+1}-(i-j)^{\gamma}(i-j-\gamma-1)\bigr),& \ i > j,\\
    0, & j=N+1;
    \end{cases}
\\
    \psi_{j, i}^{2}=
    \begin{cases}
    \frac{\delta^\gamma}{\gamma(\gamma+1)}\bigl((j-i-1)^{\gamma+1}-(j-i)^{\gamma}(j-i-\gamma-1)\bigr), & i< j,\\
    \frac{\delta^\gamma}{\gamma(\gamma+1)}\bigl((i-j+1)^{\gamma+1}-(i-j)^{\gamma}(i-j+\gamma+1)\bigr), & i \ge j,\\
    0, & j=1.
    \end{cases}
\end{gather*}
This representation is significantly more accurate than directly approximating the corresponding integrals, since most of the weight is concentrated along the singularity line $u=s$, as illustrated in Figure~\ref{fig:psi}.
\begin{figure}[htp]
    \centering
    \includegraphics[width=0.6\linewidth]{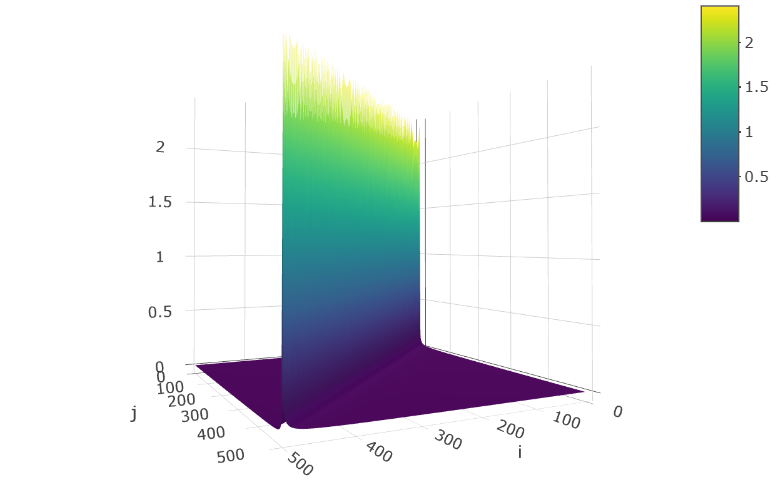}
    \caption{Surface plot for sum $\psi_{j, i}^{1}+\psi_{j, i}^{2}$ on a grid $(N+1)\times(N+1)$}
    \label{fig:psi}
\end{figure}

The system \eqref{fj_sum} can be written in matrix form as
\begin{equation}\label{linear_eq}
(\mathbf{I}_{N+1}+ \mathbf{K}_{N+1}) \textbf{H}_{N+1} = \textbf{G}_{N+1},
\end{equation}
where $\mathbf{H}_{N+1}$ and $\mathbf{G}_{N+1}$ are vectors whose components are $h_i$ and $g_i$, $i=1,\ldots,N+1$, respectively, $\mathbf{I}_{N+1}$ is the $(N+1)\times (N+1)$ identity matrix, and $\mathbf{K}_{N+1}$ is $(N+1)\times (N+1)$-matrix with components
\[
K_{i,j}= L_{i,j}\left(\psi_{j, i}^{1}+ \psi_{j, i}^{2} \right), \quad i,j=1,\ldots,N+1.
\]
The approximate solution is obtained by solving a linear system of algebraic equations \eqref{linear_eq}.

\subsection{Accuracy of the numerical approximation}
We now present a numerical example that illustrates the accuracy of the proposed method. We compare the approximate solution of equation \eqref{Fred_mle1} with the exact solution given by $h_T(u):= u$, $u\in[0,T]$.
The left-hand side of \eqref{Fred_mle1} is computed by
\[
g(u) = u + c(H_1,H_2)\int_0^T  s |u-s|^{\gamma-1} L^{(n)}(u,s)\,ds, \quad  u \in  [0, T].
\]

We construct solution $\widehat{h}(u)$, $u \in [0,1]$ using $N=500$, $H_1 =0.6$ and $H_2 = 0.7$. To reduce computational complexity, the hypergeometric functions $F_1$--$F_5$ are replaced with precomputed linear approximations on a grid of $10^5$ points. This reduces the function $L(u,s) $'s computation time by approximately 180 times on average, depending on the computatuional environment. The difference between the true value $h(u)$ and its approximation $\widehat{h}(u)$ is demonstrated in Figure~\ref{fig:log}.

\begin{figure}[htp]
    \centering
    \includegraphics[width=11cm]{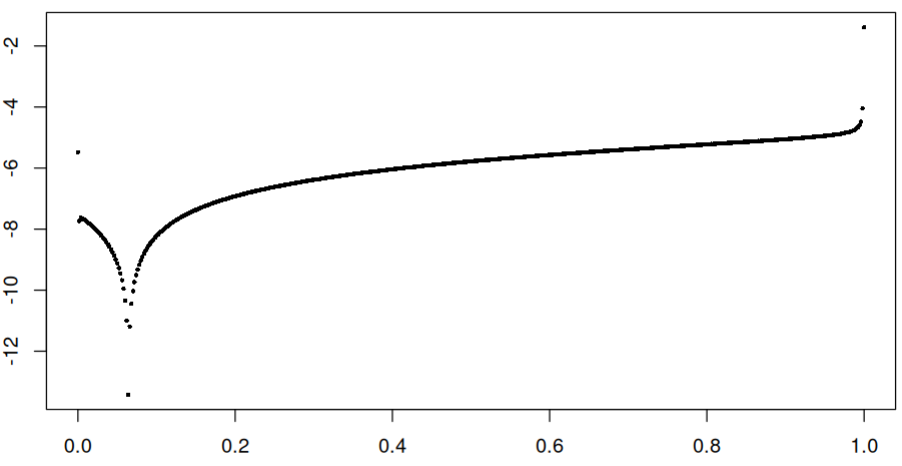}
    \caption{Natural logarithm plot of absolute difference between true $h(u)$ and an estimated solution on $[0,1]$}
    \label{fig:log}
\end{figure}

On average, the precision is approximately $10^{-6}$, with the weakest points at $u=0$ and $u=T$. In terms of computational power usage, estimating the function $g$ based on a pre-defined $h$ to evaluate the method accuracy is the most time-consuming, while constructing a solution for the pre-defined $g$ is relatively fast.
Specifically, the computation of the following integral 
\begin{equation*}
    I(u) := \int_0^T  h(s) |u-s|^{\gamma-1} L^{(n)}(u,s)\,ds, \ u \in[0,T],
\end{equation*}
without numerical methods and with an adequate precision. With the function $h(u)=u$ we get the following form of function $I(u), u \in [0,T]$ as presented in the Figure~\ref{fig:lin}.

\begin{figure}[htp]
    \centering
    \includegraphics[width=10cm]{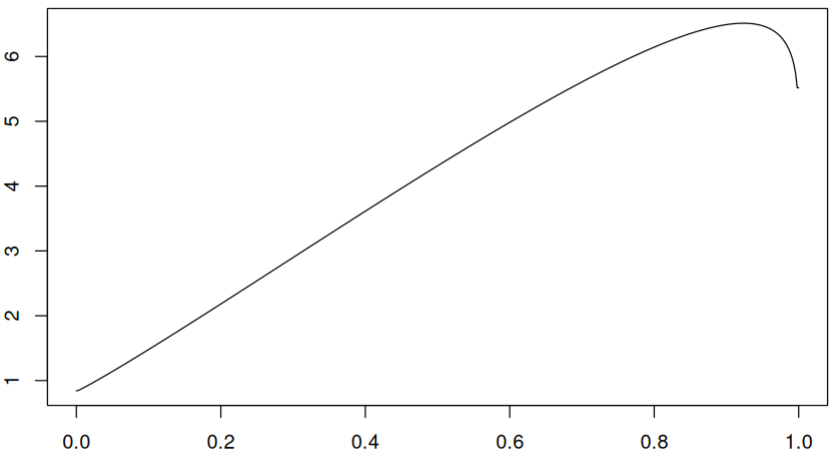}
    \caption{Integral of $h(u)=u$ with kernel $K(u,s)$ on $[0,1]$}
    \label{fig:lin}
\end{figure}

\subsection{Monte Carlo study of the estimator}
Finally, we apply the proposed method to compute the MLE $\widehat\theta_T$ defined in~\cite[Sect.~6.4.5]{MRStwofbm}, the MLE of the drift parameter $\theta$ in model \eqref{eq:themodel}. For each set of parameters $(H_1,H_2)$, we start by solving the integral Fredholm equation \eqref{Fred_mle} by the previously established numerical method, constructing an approximation of the function $h(u), u\in[0,T]$ based on a fixed $N=500$.
Then we generate 1000 trajectories of the process $X_t$ with the true value of drift parameter $\theta=1$, for a total of $\widetilde N = 100 \,T + 1$ data points, which increases with $T$. The empirical means and empirical variances of the estimates for various values of the time horizon $T$ are reported in Tables~\ref{tb1} and \ref{tb2}, respectively. Additionally, we evaluate theoretical variance of the estimator $\widehat\theta_T$ by formula \eqref{eq-prop-contt-var}. 

Analyzing Table~\ref{tb1}, we see that the estimator is clearly unbiased, with empirical variances in Table~\ref{tb2} that tend to zero, confirming the consistency of the estimator.

An important advantage of the method is that the most time-consuming step is the approximation of the function $h_T$, which is independent of the trajectories of the process.
Thus $h_T$ can be computed once for given values of $H_1$, $H_2$, and $T$ and subsequently reused for different trajectories of $X_t$.
This significantly reduces the overall computational cost when estimating the drift parameter from a large number of simulated or observed trajectories.
Another practical advantage is that the estimator is computed directly from the observed process and does not require any preliminary transformation of the data.
For each set of 1000 simulated trajectories, only one function $h_T$ was constructed.
Overall, the resulting estimator shows an accuracy comparable to that of the corresponding alternatives.
The observed differences are very small and lie within the range of Monte Carlo variability.

\begin{table}
\centering
\caption{Sample means of the estimator $\widehat\theta_T$}\label{tb1}
\begin{tabular}{lllll}
\hline
& $H_1= 0.6$, & $H_1= 0.6$, & $H_1= 0.6$, & $H_1= 0.7$,\\
$T$ & $H_2= 0.7$ & $H_2= 0.8$ & $H_2= 0.9$ & $H_2= 0.9$\\
\hline\\[-8pt]
5   & 0.98372 & 0.97900 & 0.97284 & 0.97269\\
10  & 0.97780 & 0.97505 & 0.97201 & 0.97117\\
25  & 0.99381 & 0.98929 & 0.98319 & 0.98325\\
50  & 0.99737 & 0,99432 & 0.98943 & 0.99067\\
100 & 1.00025 & 1.00117 & 1.00059 & 0.99913\\
200 & 1.01209 & 1.02222 & 1.03962 & 1.03975\\
[1pt]
\hline
\end{tabular}
\end{table}

\begin{table}
\centering
\caption{Empirical and theoretical variances of the estimator $\widehat\theta_T$}\label{tb2}
\begin{tabular}{llllll}
\hline
& Variance & $H_1= 0.6$, & $H_1= 0.6$, & $H_1= 0.6$, & $H_1= 0.7$,\\
$T$ & type & $H_2= 0.7$ & $H_2= 0.8$ & $H_2= 0.9$ & $H_2= 0.9$\\
\hline\\[-8pt]
5   & Empirical & 0.63890 & 0.78966 & 1.00644 & 1.10611\\
    & Theoretical & 0.65127 & 0.79374 & 0.99383 & 1.09431\\
10  & Empirical & 0.45875 & 0.61902 & 0.87470 & 0.97661\\
    & Theoretical & 0.40610 & 0.55083 & 0.78328 & 0.87288\\
25  & Empirical & 0.22136 & 0.35629 & 0.61758 & 0.68414\\
    & Theoretical & 0.21902 & 0.34796 & 0.59590 & 0.66281\\
50  & Empirical & 0.14286 & 0.25307 & 0.49498 & 0.54829\\
    & Theoretical & 0.13798 & 0.24963 & 0.49584 & 0.54649\\
100 & Empirical & 0.09318 & 0.18996 & 0.43055 & 0.47126\\
    & Theoretical & 0.08729 & 0.18107 & 0.41834 & 0.45552\\
200 & Empirical & 0.05856 & 0.13976 & 0.37671 & 0.40631\\
    & Theoretical & 0.05543 & 0.13254 & 0.35645 & 0.38312\\
[1pt]
\hline
\end{tabular}
\end{table}
\FloatBarrier

\appendix
\section{}

\subsection{Riemann--Liouville fractional integrals and derivatives}
\label{app:fraccalc}

\paragraph{Riemann--Liouville fractional integrals}
Let $0<\alpha<1$. For a measurable function
$f\colon[a,b]\to\mathbb{R}$, the \emph{Riemann--Liouville left- and right-sided
fractional integrals of order $\alpha$} are defined by
\[
(I^\alpha_{a+} f)(x)
= \frac{1}{\Gamma(\alpha)}\int_a^x (x-t)^{\alpha-1} f(t)\,dt,
\qquad
(I^\alpha_{b-} f)(x)
= \frac{1}{\Gamma(\alpha)}\int_x^b (t-x)^{\alpha-1} f(t)\,dt.
\]

For any $\alpha>0$ and $1\le p\le\infty$, the operators
$I^\alpha_{a+}$ and $I^\alpha_{b-}$ are bounded linear operators
$L_p(a,b)\to L_p(a,b)$, and
\[
\|I^\alpha_{a+} f\|_{L_p(a,b)}
\le
\frac{(b-a)^\alpha}{\Gamma(\alpha+1)}\,\|f\|_{L_p(a,b)},
\qquad
\|I^\alpha_{b-} f\|_{L_p(a,b)}
\le
\frac{(b-a)^\alpha}{\Gamma(\alpha+1)}\,\|f\|_{L_p(a,b)}.
\]
See \cite[(2.72)--(2.73)]{Samko}.

It is also known that $I^\alpha_{a+}$ and $I^\alpha_{b-}$ are injective operators on $L_p(a,b)$ for $1\le p\le\infty$. In particular, if $I^\alpha_{a+}f=0$ a.e. on $(a,b)$, then $f=0$ a.e.; see, e.g., \cite[Ch.~2]{Samko}.

\paragraph{Riemann--Liouville fractional derivatives}
Let $0<\alpha<1$. The \emph{Riemann--Liouville left- and right-sided
fractional derivatives} are defined by
\begin{equation}
\mathcal{D}^{\alpha}_{a+} f(x)
= \frac{1}{\Gamma(1-\alpha)}\frac{d}{dx}
\int_a^x \frac{f(t)}{(x-t)^\alpha}\,dt,
\qquad
\mathcal{D}^{\alpha}_{b-} f(x)
= -\frac{1}{\Gamma(1-\alpha)}\frac{d}{dx}
\int_x^b \frac{f(t)}{(t-x)^\alpha}\,dt.
\label{eq:derright}
\end{equation}

Denote by $I^\alpha_{a+}(L_p)$ (resp.\ $I^\alpha_{b-}(L_p)$) the class of functions $f$ that can be presented as $f=I^\alpha_{a+}\varphi$ (resp.\ $f=I^\alpha_{b-}\varphi$) for $\varphi\in L_p(a,b)$.

For $f\in I^\alpha_{a+}(L_1)$ (resp.\ $f\in I^\alpha_{b-}(L_1)$), the corresponding Riemann--Liouville fractional derivatives admit
the following \emph{Weyl representation}
for almost all
$x\in(a,b)$:
\begin{align}
\mathcal{D}^\alpha_{a+} f(x)
&=
\frac{1}{\Gamma(1-\alpha)}
\left(
\frac{f(x)}{(x-a)^\alpha}
+
\alpha\int_a^x\frac{f(x)-f(y)}{(x-y)^{\alpha+1}}\,dy
\right),
\notag\\
\mathcal{D}^\alpha_{b-} f(x)
&=
\frac{1}{\Gamma(1-\alpha)}
\left(
\frac{f(x)}{(b-x)^\alpha}
+
\alpha\int_x^b\frac{f(x)-f(y)}{(y-x)^{\alpha+1}}\,dy
\right),
\label{eq:derright-weyl}
\end{align}
see \cite[Sect.~13.1]{Samko}.

\subsection{Hypergeometric function}
\label{app:hypergeo}

The \emph{hypergeometric function} $\hyper(a,b;c;z)$ is defined
for real parameters $a$, $b$, $c$ and real argument $z$.
Throughout the paper we assume $c>b>0$ and $z<1$.
In this case $\hyper(a,b;c;z)$ admits the Euler integral representation
\begin{equation}\label{eq:hypergeom}
\hyper(a,b;c;z)
=
\frac{\Gamma(c)}{\Gamma(b)\Gamma(c-b)}
\int_0^1
t^{\,b-1}(1-t)^{c-b-1}(1-zt)^{-a}\,dt,
\end{equation}
which defines an analytic function of $z$ on $(-\infty,1)$,
see \cite[Eq.~15.3.1]{Abramowitz-Stegun}.

In this paper we only use the case $z\in[0,1)$.
At $z=0$ the hypergeometric function is continuous and equals $1$.

The following transformation formulas can be derived from the representation \eqref{eq:hypergeom}
(see \cite[15.3.3 and 15.3.6]{Abramowitz-Stegun}):
\begin{equation}\label{eq:AS-15-3-3}
{}_2F_1(a,b;c;z)
=
(1-z)^{\,c-a-b}\,
{}_2F_1(c-a,c-b;c;z)
\end{equation}
and
\begin{multline}\label{eq:hypergeo-connection}
\hyper(a,b;c;z)
=\frac{\Gamma(c)\Gamma(c-a-b)}{\Gamma(c-a)\Gamma(c-b)}
\,\hyper(a,b;a+b-c+1;1-z)
\\
{}+\frac{\Gamma(c)\Gamma(a+b-c)}{\Gamma(a)\Gamma(b)}
(1-z)^{\,c-a-b}
\,\hyper(c-a,c-b;c-a-b+1;1-z).
\end{multline}

\paragraph{Behavior as $z\uparrow1$}
Let
\[
\lambda:=c-a-b.
\]
The behavior of $\hyper(a,b;c;z)$ as $z\uparrow1$ is governed by the sign of $\lambda$,
see \cite[\S15.1]{Abramowitz-Stegun}.
In particular, in the case $\lambda>0$, $\hyper(a,b;c;z)$ is finite at $z=1$ and
\begin{equation}\label{eq:hypergeo-limit-1}
\hyper(a,b;c;1)
=
\frac{\Gamma(c)\Gamma(c-a-b)}{\Gamma(c-a)\Gamma(c-b)},
\end{equation}
see \cite[Eq.~15.1.20]{Abramowitz-Stegun}.
Moreover, by \eqref{eq:hypergeo-connection}, if $\lambda>0$ then
\[
\hyper(a,b;c;z)
=
\hyper(a,b;c;1)
+
O\!\left((1-z)^{\lambda\wedge1}\right),
\qquad z\uparrow1.
\]
Since $\hyper(a,b;c;z)$ is analytic on $[0,1)$, it follows that
$\hyper(a,b;c;\cdot)$ extends to a H\"older continuous function on $[0,1]$
with exponent $\lambda$ when $0<\lambda<1$, and to a Lipschitz continuous
function on $[0,1]$ when $\lambda\ge1$.

\paragraph{Differentiation}
The hypergeometric function satisfies
\begin{equation}\label{eq:hypergeo-deriv-app}
\frac{d}{dz}\,\hyper(a,b;c;z)
=
\frac{ab}{c}\,
\hyper(a+1,b+1;c+1;z),
\end{equation}
see \cite[Eq.~15.2.1]{Abramowitz-Stegun}.

\subsection{Some integrals related to beta-function}
Recall that for $a>0$ and $b>0$ the beta-function admits the following integral representations
\begin{equation}\label{eq:beta}
B(a,b)
=\int_0^1 y^{a-1} (1-y)^{b-1}dy
=\int_0^\infty x^{a-1}(1+x)^{-a-b}dx.
\end{equation}
\begin{lemma}\label{lem:beta-integrals}
Let $\mu\in(0,1)$ and $\nu\in(0,1)$. Then 
\begin{enumerate}[(i)]
\item
$\displaystyle\int_0^1 y^{-\mu-1}\bigl((1-y)^{-\nu}-1\bigr)dy
=\frac{\mu+\nu-1}{\mu}B(1-\mu,1-\nu)+\frac{1}{\mu}$,

\item
$\displaystyle\int_{0}^{\infty} x^{-\mu-1}\bigl(1-(1+x)^{-\nu}\bigr)\,dx
=
\frac{\nu}{\mu} B(1-\mu,\mu+\nu)$.
\end{enumerate}
\end{lemma}

\begin{proof}
$(i)$
With the substitution $x=\frac{y}{1-y}$, i.e.\ 
$y=\frac{x}{1+x}$ and $dy=(1+x)^{-2}dx$, the integral becomes
\[
I \coloneqq \int_0^1 y^{-\mu-1}\bigl((1-y)^{-\nu}-1\bigr)dy
= \int_0^\infty x^{-\mu-1}
\bigl((1+x)^{\mu+\nu-1}-(1+x)^{\mu-1}\bigr)dx .
\]
Let $f(x)=(1+x)^{\mu+\nu-1}-(1+x)^{\mu-1}$. Since
$f(x)\sim \nu x$ as $x\downarrow0$ and
$x^{-\mu}f(x)\to0$ as $x\to\infty$, integration by parts with
$d(x^{-\mu})=-\mu x^{-\mu-1}dx$ gives
\[
I = \frac{1}{\mu}\int_0^\infty x^{-\mu}f'(x)\,dx .
\]
Since
\[
f'(x) = (\mu+\nu-1)(1+x)^{\mu+\nu-2} - (\mu-1)(1+x)^{\mu-2},
\]
the standard beta-integral identity \eqref{eq:beta}
yields
\[
I = \frac{\mu+\nu-1}{\mu}B(1-\mu,1-\nu)
-\frac{\mu-1}{\mu}B(1-\mu,1).
\]
Since $B(1-\mu,1)=\frac{1}{1-\mu}$, the claim follows.

\medskip
$(ii)$
Using the representation
\[
1-(1+x)^{-\nu}
=\nu \int_0^x (1+y)^{-\nu-1}dy
=\nu x\int_0^1(1+r x)^{-\nu-1}dr
\]
and the Fubini theorem, we get
\[
\int_0^\infty x^{-\mu-1}\bigl(1-(1+x)^{-\nu}\bigr)dx
=
\nu\int_0^1\!\!
\int_0^\infty x^{-\mu}(1+r x)^{-\nu-1}dx\,dr .
\]
The substitution $y=r x$ gives
\[
\int_0^\infty x^{-\mu}(1+r x)^{-\nu-1}dx
=
r^{\mu-1}B(1-\mu,\mu+\nu),
\]
and integration in $r$ yields the result.
\end{proof}

\end{document}